\documentclass[11pt]{article}

\usepackage[margin=1.25in]{geometry}

\usepackage{amssymb,amsmath}
\usepackage{amsthm}
\usepackage{hyperref}
\usepackage{mathrsfs}
\usepackage{textcomp}
\usepackage{enumerate}
\usepackage{framed}
\usepackage{ulem}
\usepackage{color,xcolor}

\usepackage{soul}
\soulregister\cite7 % 针对\cite命令
\soulregister\citep7 % 针对\citep命令
\soulregister\citet7 % 针对\citet命令
\soulregister\ref7 % 针对\ref命令
\soulregister\pageref7 % 针对\pageref命令
\soulregister\eqref7

\hypersetup{
    colorlinks,%
    citecolor=blue,%
    filecolor=blue,%
    linkcolor=blue,%
    urlcolor=black
}

\newtheorem{theorem}{Theorem}[section]
\newtheorem{definition}{Definition}[section]
\newtheorem{lemma}{Lemma}[section]
\newtheorem{remark}{Remark}[section]
\newtheorem{proposition}{Proposition}[section]
\newtheorem{corollary}{Corollary}[section]

\numberwithin{equation}{section}

\makeatletter

\newdimen\bibspace
\renewenvironment{thebibliography}[1]{%
 \section*{\refname %or \bibname if you use ``book'' as the documentclass
       \@mkboth{\MakeUppercase\refname}{\MakeUppercase\refname}}%
     \list{\@biblabel{\@arabic\c@enumiv}}%
          {\settowidth\labelwidth{\@biblabel{#1}}%
           \leftmargin\labelwidth
           \advance\leftmargin\labelsep
           \itemsep\bibspace
           \parsep\z@skip     %
           \@openbib@code
           \usecounter{enumiv}%
           \let\p@enumiv\@empty
           \renewcommand\theenumiv{\@arabic\c@enumiv}}%
     \sloppy\clubpenalty4000\widowpenalty4000%
     \sfcode`\.\@m}
    {\def\@noitemerr
      {\@latex@warning{Empty `thebibliography' environment}}%
     \endlist}

\makeatother

           \newcommand{\ud}{\mathrm{d}}
\newcommand{\be}{\begin{equation}}      \newcommand{\ee}{\end{equation}}

\begin{document}

\title{Uniqueness of positive solutions to the higher order Brezis-Nirenberg problem}

\author{Zhongwei Tang\thanks{Z. Tang is supported by National Natural Science Foundation of China (12071036).},\, Ning Zhou}

\date{}

\maketitle

\begin{abstract}
In this paper, we study the higher order Brezis-Nirenberg problem under the Navier boundary condition
\be\label{eq}
\begin{cases}
(-\Delta)^m u=\varepsilon u+u^{p} & \text { in }\, \Omega, \\
u>0 & \text { in }\, \Omega, \\
u=-\Delta u=\cdots=(-\Delta)^{m-1} u=0 & \text { on }\, \partial \Omega,
\end{cases}
\ee
where $\Omega$ is a strictly convex smooth bounded domain in $\mathbb{R}^n$ with $n \geq 4m$, $m \in \mathbb{N}_{+}$, $\varepsilon\in (0,\lambda_{1})$, $\lambda_{1}$ is the first Navier eigenvalue for $(-\Delta)^{m}$ in $\Omega$, and $p=\frac{n+2m}{n-2m}$. We prove that the solutions of \eqref{eq} are unique if either $\varepsilon$ close to $\lambda_1$ or $\varepsilon$ close to 0 and $\Omega$ satisfies some symmetry assumptions. The proof is mainly based on our previous works about the blow up analysis and compactness result for solutions to higher order critical elliptic equations and the asymptotic behavior of solutions to \eqref{eq}.
\end{abstract}


\section{Introduction}
In a celebrated paper, Brezis-Nirenberg \cite{BNPositive1983} studied the following nonlinear critical elliptic partial differential equation
\be\label{eq:BN}
\begin{cases}
-\Delta u-\lambda u=|u|^{2^*-2} u & \text { in }\, \Omega, \\
u=0 & \text { on }\, \partial \Omega,
\end{cases}
\ee
where $\Omega$ is a smooth bounded domain in $\mathbb{R}^n$ for $n \geq 3$, $2^*=\frac{2n}{n-2}$ is the critical Sobolev exponent associated to critical embedding of the space $H_0^1(\Omega)$.

Denoting by $\lambda_1(\Omega)$ the first eigenvalue of $-\Delta$ on $\Omega$ under Dirichlet boundary condition, they proved that problem \eqref{eq:BN} admits at least one positive solution for any $\lambda \in(0, \lambda_1(\Omega))$ provided that $n\geq 4$. They also proved that a solution gap phenomenon appears in the case $n=3$, to be precise, there exists a constant $\lambda_* \in(0, \lambda_1(\Omega))$ such that, for any $\lambda \in(\lambda_*, \lambda_1(\Omega))$, problem \eqref{eq:BN} admits at least one positive solution. On the other hand, problem \eqref{eq:BN} admits no positive solution for $n \geq 3$ if either $\lambda \geq \lambda_1(\Omega)$ or $\lambda \leq 0$ and $\Omega$ is a star-shaped domain.

There have been tremendous amount of works in related to the existence, nonexistence, and multiplicity of solutions to \eqref{eq:BN} over the past decades. Concerning the uniqueness problem of \eqref{eq:BN}, the shape of the domain $\Omega$ is important. For example, if $\Omega$ is a ball in $\mathbb{R}^n$, it was proved in \cite{ZUniqueness1992,SUniqueness1993,AYAn1994} that the solutions of \eqref{eq:BN} is unique if $\varepsilon\in(0,\lambda_1(\Omega))$, $n\geq 4$ and $\varepsilon\in(\lambda_1(\Omega)/4,\lambda_1(\Omega))$, $n=3$. On the other hand, it was proved in \cite{LSolutions1992} for $n=4$ and in \cite{RA1989} for $n \geq 5$ that \eqref{eq:BN} has at least $\operatorname{cat}_{\Omega}(\Omega)$ solutions for $\varepsilon$ small, where $\operatorname{cat}_X(Y)$ denotes the Lusternik-Schnirelman category of $Y$ in $X$. Zhu \cite{ZUniqueness2000} obtained the uniqueness results about \eqref{eq:BN} when $\varepsilon$ close to $\lambda_1(\Omega)$. If the region $\Omega$ satisfies some assumptions, Cerqueti and Grossi (see \cite{CA1999,CGLocal2001}) proved the uniqueness of solutions to \eqref{eq:BN} when $\varepsilon>0$ small. Recently, Cao-Luo-Peng \cite{CLPThe2021} studied the local uniqueness of multi-peak solutions and the exact number of positive solutions to \eqref{eq:BN} when $\varepsilon>0$ small.

Some of the aforementioned results were extend to the biharmonic version of the Brezis-Nirenberg problem under the Navier boundary condition
\be\label{eq:Navier}
\begin{cases}
(-\Delta)^2 u=\varepsilon u+u^{p} & \text { in }\, \Omega, \\
u>0 & \text { in }\, \Omega, \\
u=-\Delta u=0 & \text { on }\, \partial \Omega,
\end{cases}
\ee
where $p=\frac{n+4}{n-4}$. Let $\lambda_{1,2}$ be the first eigenvalue of $(-\Delta)^2$ on $\Omega$ under Navier boundary condition, it was proved by Van der Vorst in \cite{VBest1995} that if $n \geq8$, \eqref{eq:Navier} admits at least one solution for any $\varepsilon\in(0,\lambda_{1,2})$; if $n \in\{5,6,7\}$, \eqref{eq:Navier} admits at least one solution for $\varepsilon<\lambda_{1,2}$ close to $\lambda_{1,2}$; \eqref{eq:Navier} admits no solution for $\varepsilon \geq \lambda_{1,2}$ and, if $\Omega$ is star-shaped, for $\varepsilon \leq 0$. Similar conclusions hold also for the higher order versions of the Brezis-Nirenberg problem under the Navier boundary condition
\be\label{eq:higherorder}
\begin{cases}
(-\Delta)^{m} u=\varepsilon u+u^{p} & \text { in }\, \Omega, \\
u>0 & \text { in }\, \Omega,\\
u=-\Delta u=\cdots=(-\Delta)^{m-1} u=0 & \text { on }\, \partial \Omega,
\end{cases}
\ee
where $p=\frac{n+2m}{n-2m}$, $n> 2m$, $m\in \mathbb{N}_+$, the readers can refer to \cite{GGSPolyharmonic2010}. Let $\lambda_{1,m}$ be the first eigenvalue of $(-\Delta)^m$ on $\Omega$ under Navier boundary condition, when there is no confusion, we just denote it by $\lambda_1$.

In this paper we shall address the uniqueness of solutions to \eqref{eq:higherorder} as $\varepsilon$ close to $\lambda_1$ or 0. Our first result deals with the case $\varepsilon$ close to $\lambda_1$.

\begin{theorem}\label{thm:1}
Let $\Omega$ be a strictly convex smooth bounded domain in $\mathbb{R}^{n}$ with $n \geq 4m$, $m\in \mathbb{N}_+$. Then there exists a constant $\varepsilon_{1}>0$ such that, for any $\varepsilon \in(\lambda_{1}-\varepsilon_{1}, \lambda_{1})$, \eqref{eq:higherorder} has a unique solution.
\end{theorem}

Unlike the case when $\varepsilon$ close to $\lambda_1$, the uniqueness of the solutions to \eqref{eq:higherorder} is closely related to the shape of the domain when $\varepsilon>0$ small. For example, it was shown in \cite{ESConcentration2006} that if $m=2$ and $\Omega$ has a rich topology, described by its Lusternik-Schnirelmann category, then \eqref{eq:Navier} has multiple solutions, at least as many as $\operatorname{cat}_{\Omega}(\Omega)$, in case $\varepsilon>0$ is sufficiently small.

Before stating our second main result, a symmetry assumption about the domain $\Omega$ needs to be introduced:

$(A)$ $\Omega$ is symmetric with respect to the hyperplanes $\{x_{i}=0\}$, $i=1, \cdots, n$.

Our second result of this paper reads as follows:

\begin{theorem}\label{thm:2'}
Let $\Omega$ be a strictly convex smooth bounded domain in $\mathbb{R}^{n}$ with $n \geq 6m$, satisfying $(A)$. Then there exists a constant $\varepsilon_{2}>0$ such that, for any $\varepsilon \in(0, \varepsilon_{2})$, \eqref{eq:higherorder} has a unique solution.
\end{theorem}

For the convenience of calculation, we only prove the case of $m=2$, i.e., we only prove the following theorem, the proof for any $m\in \mathbb{N}_+$ just need to make some minor modifications.

\begin{theorem}\label{thm:2}
Let $\Omega$ be a strictly convex smooth bounded domain in $\mathbb{R}^{n}$ with $n \geq 12$, satisfying $(A)$. Then there exists a constant $\varepsilon_{3}>0$ such that, for any $\varepsilon \in(0, \varepsilon_{3})$, \eqref{eq:Navier} has a unique solution.
\end{theorem}

\begin{remark}
Let's briefly explain the reason for the dimension limitations in the above theorem. First, in the proof of Theorem \ref{thm:1}, we take advantage of the compactness result of solutions to higher order elliptic equation obtained in \cite{NTZCompactness2022}, where $n\geq 4m$ is required, see Proposition \ref{pro:IMRN}. On the other hand, in the proof of Theorem \ref{thm:2}, we use the Pohozaev identity to get an estimate of $\varepsilon$ (see Lemma \ref{lem:vare-muvare}), and need to use this estimate to get an $L^{\infty}$ estimate in \eqref{eq:dimention}, that's the reason why we limit $n\geq12$.
\end{remark}

Let us briefly outline the proof of our main results.

Let $u_{\varepsilon}$ be a solution of \eqref{eq:higherorder}. A crucial ingredient in the proof of Theorem \ref{thm:1} is to show that $\|u_{\varepsilon}\|_{L^{\infty}(\Omega)}$ are uniformly bounded as $\varepsilon \rightarrow \lambda_1$. First of all, by using the method of moving planes, we can show that there exist some positive constants $\delta_{0}$ and $C=C(\delta_{0})$, such that for all $\varepsilon \in(\lambda_{1} / 2, \lambda_{1})$,
$$
u_{\varepsilon}(x) \leq C, \quad \forall\, x \in\{y \in \Omega \mid \operatorname{dist}(y, \partial \Omega) \leq \delta_{0}\}.
$$
On the other hand, in order to show that $u_{\varepsilon}$ are uniformly bounded in $\{y \in \Omega \mid \operatorname{dist}(y, \partial \Omega) \geq \delta_{0}\}$, we make use of a compactness result for higher order elliptic equations obtained by the authors and Niu, see Proposition \ref{pro:IMRN}.

To prove Theorem \ref{thm:2}, we first use the blow up analysis technique for higher order elliptic equations obtained in \cite{NTZCompactness2022} to prove that if $\Omega$ satisfies assumption $(A)$, then $u_{\varepsilon}$ satisfies
\be\label{eq:leastenergy}
\lim _{\varepsilon \rightarrow 0} \frac{\int_{\Omega}|\Delta u_{\varepsilon}|^2 \,\ud x}{(\int_{\Omega}|u_{\varepsilon}|^{p+1}\,\ud x)^{2 / (p+1)}}=S_n,
\ee
where $S_n$ is the best constant in the Sobolev inequality
\be\label{eq:Sobolevine}
S_n\Big(\int_{\mathbb{R}^n}|u|^{p+1} \,\ud x\Big)^{2 / (p+1)} \leq  \int_{\mathbb{R}^n} |\Delta u|^2 \,\ud x,\quad u\in D^{2, 2}(\mathbb{R}^n),
\ee
and $D^{2, 2}(\mathbb{R}^n)$ is the completion of $C_0^{\infty}(\mathbb{R}^n)$ under the norm
$$
\|u\|_{D^{2, 2}(\mathbb{R}^n)}=\Big(\int_{\mathbb{R}^n}|\Delta u|^2 \,\ud x\Big)^{1 / 2}.
$$
It was shown in \cite{EFJCritical1990,LSharp1983} that
\be\label{eq:constantSn}
S_n=\pi^2 n(n-4)(n^2-4) \Big(\frac{\Gamma(\frac{n}{2})}{\Gamma(n)}\Big)^{{4}/{n}},
\ee
where $\Gamma$ is the Gamma function. Next, we study the asymptotic behavior of positive solutions to \eqref{eq:Navier} under the assumption \eqref{eq:leastenergy}. This asymptotic behavior results can be seen as a analog of the results of Han \cite{HAsymptotic1991}, Chou-Geng \cite{CGAsymptotics2000}, and Choi-Kim-Lee \cite{CKLAsymptotic2014}, which is interesting in itself. Finally, we use the asymptotic behavior to prove the uniqueness of solutions to \eqref{eq:Navier}.

This paper is organized as follows: in Section 2, we will give some preliminaries and useful known facts; Section 3 will be devoted to the proof of Theorem \ref{thm:1}; Section 4 will focus on the asymptotic behavior of solutions to \eqref{eq:Navier} as $\varepsilon\to 0$; Section 5 will give the proof of Theorem \ref{thm:2}.

\section{Preliminaries}

In this section, we shall outline some known results, playing a key role in our proofs in the subsequent sections. We denote $B_R(x)$ as the ball in $\mathbb{R}^{n}$ with radius $R$ and center $x$ and $B_R$ as the ball in $\mathbb{R}^n$ with radius $R$ and center 0.

\subsection{Blow up analysis for higher order elliptic equations}

The blow up analysis technique was introduced first by R. Schoen for the prescribed scalar curvature problem. Later, these ideas had been developed by several authors and applied to the Yamabe problem, the Nirenberg problem, and the fractional Nirenberg problem, see, for example, \cite{LPrescribing1995,LZYamabe1999,JLXOn2014,JLXThe2017,NPXCompactness2018,LXCompactness2019}. In this subsection, we recall some results of blow up analysis for higher order elliptic equations obtained by the authors and Niu in \cite{NTZCompactness2022}. In the following, we take $m=2$ for conveniences. One can replace it by any integer.

Let $\Omega$ be an smooth bounded domain in $\mathbb{R}^n$, $n \geq 8$, $p=\frac{n+4}{n-4}$, $0<\varepsilon_{i} <\lambda_1$, and $\{u_{i}\}$ be a sequence of $C^4$ functions satisfying
\be\label{eq:ui}
\begin{cases}
(-\Delta)^2 u_i=\varepsilon_i u_i+u_i^{p} & \text { in }\, \Omega, \\
u_i>0 & \text { in }\, \Omega, \\
u_i=-\Delta u_i=0 & \text { on }\, \partial \Omega.
\end{cases}
\ee
Let $\Omega^{\prime}$ be a compact subset of $\Omega$. We say a point $\bar{x} \in \Omega^{\prime}$ is a blow up point of $\{u_i\}$ if $u_i(x_i) \rightarrow \infty$ for some $x_i \rightarrow \bar{x}$.

\begin{definition}\label{def:isolated}
Let $\{u_i\}$ satisfy \eqref{eq:ui}. A point $\bar{x} \in \Omega^{\prime}$ is called an isolated blow up point of $\{u_i\}$ if there exist $0<\bar{r}<\operatorname{dist}(\bar{x}, \partial \Omega)$, $C>0$, and a sequence $x_i$ tending to $\bar{x}$, such that $x_i$ is a local maximum of $u_i$, $u_i(x_i) \rightarrow \infty$ and
$$
u_i(x) \leq C|x-x_i|^{-(n-4) /2} \quad \text { for all }\, x \in B_{\bar{r}}(x_i).
$$
\end{definition}

Let $x_i \rightarrow \bar{x}$ be an isolated blow up point of $u_i$, we set
$$
\bar{u}_i(r)=\frac{1}{|\partial B_r(x_i)|} \int_{\partial B_r(x_i)} u_i\,\ud s,\quad 0<r<\bar{r},
$$
and
$$
\bar{w}_i(r)=r^{(n-4) /2} \bar{u}_i(r),\quad 0<r<\bar{r}.
$$

\begin{definition}\label{def:isolatedsimple}
We say $x_i \rightarrow \bar{x} \in \Omega^{\prime}$ is an isolated simple blow up point if $x_i \rightarrow \bar{x}$ is an isolated blow up point, such that for some $\rho \in(0, \bar{r})$ (independent of $i$), $\bar{w}_i$ has precisely one critical point in $(0, \rho)$ for large $i$.
\end{definition}

\begin{lemma}\label{lem:IMRN-Harnack}
Let $\{u_i\}$ satisfy \eqref{eq:ui} and $x_i \rightarrow 0$ is an isolated blow up point of $\{u_i\}$, that is, for some positive constants $A_1$ and $\bar{r}$ independent of $i$,
$$
u_i(x) \leq A_1|x-x_i|^{-(n-4)/2}  \quad \text { for all }\, x \in B_{\bar{r}}(x_i) \subset \Omega.
$$
Then for any $0<r<\bar{r} / 3$, we have the following Harnack inequality:
$$
\sup_{B_{2 r}(x_i) \backslash \overline{B_{r / 2}(x_i)}}  u_i \leq C\inf_{B_{2 r}(x_i) \backslash \overline{B_{r / 2}(x_i)}}  u_i,
$$
where $C>0$ depends only on $n$, $A_1$, and $\bar{r}$.
\end{lemma}

\begin{lemma}\label{lem:IMRN-converge}
Suppose that the hypotheses of Lemma \ref{lem:IMRN-Harnack} hold. Then for any $R_i \rightarrow \infty$ and $\varepsilon_i \rightarrow 0^{+}$, we have, after passing to a subsequence (still denoted as $\{u_i\}$, $\{x_i\}$, etc.), that
$$
\|m_i^{-1} u_i(m_i^{-2 / (n-4)} \cdot+x_i)-(1+\bar{c}|\cdot|^2)^{(4-n) / 2}\|_{C^2(B_{2 R_i})} \leq \varepsilon_i,
$$
$$
r_i:=R_i m_i^{-2/ (n-4)} \rightarrow 0 \quad \text { as }\,  i \rightarrow \infty,
$$
where $m_i=u_i(x_i)$ and $\bar{c}>0$ depends only on $n$.
\end{lemma}

\begin{lemma}\label{lem:IMRN-sharp}
Under the hypotheses of Lemma \ref{lem:IMRN-Harnack}, and in addition that $x_i \rightarrow 0$ is also an isolated simple blow up point with the constant $\rho$, we have
$$
u_i(x) \leq C u_i(x_i)^{-1}|x-x_i|^{4-n} \quad \text { for all }\, |x-x_i| \leq 1,
$$
where $C > 0$ depends only on $n$, $A_1$, and $\rho$.
\end{lemma}

\begin{lemma}\label{lem:IMRN-isomustisosimle}
Suppose that the hypotheses of Lemma \ref{lem:IMRN-Harnack} hold. Then if $n\geq 10$, after passing to a subsequence, $x_i \rightarrow 0$ is an isolated simple blow up point of $\{u_i\}$.
\end{lemma}

\subsection{Pohozaev identity}

In this subsection, we present some Pohozaev identities, these identities are crucial in the proof of the asymptotic behavior of positive solutions to \eqref{eq:Navier}.

\begin{lemma}[Pohozaev identity]\label{lem:Pohozaev}
Let $\Omega$ be an smooth bounded domain in $\mathbb{R}^n$ with $n \geq 9$, and let $u$ be a solution of \eqref{eq:Navier}. Then the following identity holds:
\be\label{eq:Pohozaev}
2 \varepsilon \int_{\Omega} u^{2} \,\ud x=\int_{\partial \Omega} \frac{\partial u}{\partial \nu}\frac{\partial (-\Delta u)}{\partial \nu}(x \cdot \nu) \,\ud s,
\ee
where $\nu$ is the unit outward normal vector of $\partial \Omega$.
\end{lemma}

\begin{proof}
The proof can be found in \cite[Theorem 7.29]{GGSPolyharmonic2010}.
\end{proof}

Let $G=G(x, y)$ denote the Green function of $(-\Delta)^2$ under the Navier boundary condition:
$$
\begin{cases}
(-\Delta)^2 G(\cdot, y)=\delta_y(\cdot) & \text { in }\, \Omega, \\
G(\cdot, y)=-\Delta G(\cdot, y)=0 & \text { on }\, \partial \Omega,
\end{cases}
$$
where $\delta_y$ is the Dirac function. Decompose $G$ as $G(x, y)=\Gamma(x, y)+g(x, y)$, where $\Gamma(x, y)$ is the fundamental solution of $(-\Delta)^2$ on $\mathbb{R}^n$ defined as
\be\label{eq:fundsolu}
\Gamma(x, y)=
\begin{cases}
\sigma_{n}|x-y|^{4-n}, & n>4, \\
\displaystyle\sigma_4 \log \frac{1}{|x-y|}, & n=4,
\end{cases}
\ee
where
$$
\sigma_{n}=\frac{1}{2(n-2)(n-4) \omega_n},\quad n>4;\quad \sigma_4=\frac{1}{4 \omega_{4}},
$$
and $\omega_n$ is the volume of the $(n-1)$ dimensional unit sphere in $\mathbb{R}^n$. $g(x, y) \in C^{\infty}(\Omega \times \Omega)$ is called the regular part of the Green function, and satisfies
\be\label{eq:regupart}
\begin{cases}
(-\Delta)^2 g(\cdot, y)=0 & \text { in }\, \Omega,  \\
g(\cdot, y)=-\Gamma(\cdot,y) & \text { on }\, \partial\Omega,\\
\Delta g(\cdot, y)=-\Delta \Gamma(\cdot,y) & \text { on }\, \partial\Omega.
\end{cases}
\ee
For any $y \in \Omega$, we denote $R(y):=g(y, y)$, which is called the Robin function of the Green function of $(-\Delta)^2$ with the Navier boundary condition.

\begin{lemma}[Pohozaev identity for the Green function]\label{lem:PohozaevGreen}
Let $\Omega$ be an smooth bounded domain in $\mathbb{R}^n$ with $n \geq 9$. For any $y \in \Omega$,
$$
\int_{\partial \Omega}\frac{\partial G(x,y)}{\partial \nu} \frac{\partial (-\Delta_xG(x,y))}{\partial \nu}(x \cdot \nu)  \,\ud s_{x}=-\frac{n-4}{2} R(y),
$$
where $\nu$ is the unit outward normal vector of $\partial \Omega$.
\end{lemma}

\begin{proof}
The proof can be found in \cite[Theorem 3.1]{TSome2013}.
\end{proof}

\subsection{Some estimates}

In this subsection, we present some elliptic estimates that will be used frequently in our proof.

\begin{lemma}\label{lem:localbound}
Let $u$ be a weak solution of
$$
(-\Delta)^{2} u=a(x) u \quad \text { in }\, \Omega,
$$
where $a \in L_{loc}^{\alpha}(\Omega)$ with $\alpha>n / 4$. Then for any $q >0$, there exist $C=C(q)>0$ and $R>0$, such that for any $0<r<R$ and $y \in \mathbb{R}^{n}$, we have
$$
\sup _{\Omega \cap B_r(y)}|u(x)| \leq C\Big(\frac{1}{r^{n}} \int_{\Omega \cap B_{2 r}(y)}|u(x)|^{q} \,\ud x\Big)^{1 /q}.
$$
\end{lemma}

\begin{proof}
The proof can be found in \cite[Theorem 4.9, Definition 2.2, Example 2.4]{CMHarnack2006}.
\end{proof}

\begin{lemma}\label{lem:esti-unablau}
Let $u$ satisfies
$$
\begin{cases}
-\Delta u=f & \text { in }\, \Omega, \\
u=0 & \text { on }\, \partial \Omega.
\end{cases}
$$
Let $\omega^{\prime} \subset \subset \omega$ be two neighborhoods of $\partial \Omega$. Then
$$
\|u\|_{W^{1, q}(\Omega)} \leq C\|f\|_{L^{1}(\Omega)}
$$
and
$$
\|u\|_{C^{1, \alpha}(\omega^{\prime})} \leq C(\|f\|_{L^{1}(\Omega)}+\|f\|_{L^{\infty}(\omega)})
$$
hold for $q\in [1, \frac{n}{n-1})$, $\alpha \in(0,1)$. Here $C$ is independent of $u$ and $f$.
\end{lemma}

\begin{proof}
The proof can be found in \cite[Lemma 2]{HAsymptotic1991}.
\end{proof}

\section{Proof of Theorem \ref{thm:1}}

In this section, we give the proof of Theorem \ref{thm:1}.

In \cite{NTZCompactness2022}, the authors and Niu used blow up analysis to obtain the following compactness result for the solutions of higher order critical elliptic equations:

\begin{proposition}[\cite{NTZCompactness2022}]\label{pro:IMRN}
Suppose that $1 \leq m \leq n / 4$ and $m$ is an integer. Let $u \in C^{2 m}(B_3)$ be a nonnegative solution of
$$
(-\Delta)^m u-a(x) u=u^{\frac{n+2 m}{n-2 m}} \quad \text { in }\, B_3,
$$
where $a(x)$ is a nonnegative smooth function in $B_3$. Suppose
$$
(-\Delta)^k u \geq 0 \quad \text { in }\, B_3, \quad k=1, \cdots, m-1.
$$
If either
\begin{enumerate}[(i)]
  \item
$a>0$ in $B_2$ and $n \geq 4 m$, or
  \item
$\Delta a>0$ in $\{x \mid a(x)=0\} \cap B_2$ and $n \geq 4 m+2$
\end{enumerate}
holds, then
$$
\|u\|_{C^{2 m}(B_1)} \leq C,
$$
where $C>0$ depends only on $n, m,\|a\|_{C^4(B_3)}$ and $\inf _{B_2} a$ if $(i)$ holds; otherwise, it depends only on $n, m,\|a\|_{C^4(B_3)}$ and $\inf _{\{x \mid a(x)=0\} \cap B_2} \Delta a$.
\end{proposition}

The key ingredient in the proof of Theorem \ref{thm:1} is the following a priori estimates.

\begin{proposition}\label{pro:uinfleqC}
There exist some positive constants $\delta$ and $C=C(\delta)$ such that, if $u_{\varepsilon}$ is a solution of \eqref{eq:higherorder} for $\varepsilon \in(\lambda_1-\delta, \lambda_1)$, then
$$
\|u_{\varepsilon}\|_{L^{\infty}(\Omega)} \leq C.
$$
\end{proposition}

\begin{proof}
Firstly, we prove that there exist some positive constants $\bar{\delta}$ and $C=C(\bar{\delta})$ such that, for all $\varepsilon \in(0, \lambda_{1})$,
\be\label{eq:uvare<CinOmega1}
u_{\varepsilon}(x) \leq C, \quad \forall\, x \in\Omega_1:=\{y \in \Omega \mid \operatorname{dist}(y, \partial \Omega) \leq \bar{\delta}\}.
\ee
Let $u_{1}(x)=u_{\varepsilon}(x)$, $u_{2}(x)=-\Delta u_{\varepsilon}(x)$, $\cdots$, $u_{m}(x)=(-\Delta)^{m-1} u_{\varepsilon}(x)$. Then $(u_{1},\cdots, u_{m})$ is a solution of the following elliptic system
$$
\begin{cases}
-\Delta u_{i}=u_{i+1},\quad i=1,\cdots,m-1 & \text { in }\, \Omega,\\
-\Delta u_{m}=\varepsilon u_{1}+u_{1}^p & \text { in }\, \Omega,\\
u_i>0,\quad i=1,\cdots,m & \text { in }\, \Omega, \\
u_{1}=\cdots=u_{m}=0 & \text { on }\, \partial\Omega,
\end{cases}
$$
where we have used the maximum principle. Since the domain $\Omega$ is strictly convex, by applying the method of moving planes for systems of elliptic equations (see, for example, \cite[Lemma 4.1]{TSymmetry1981}), we get that, for any $x_0 \in \partial \Omega$, there exists $\delta_0=\delta_0(x_0,\Omega)>0$ such that $u_{\varepsilon}(x)$ is monotone increasing along the internal normal direction in the region
$$
\overline{\Sigma_{\delta_0}}:=\{x \in \overline{\Omega} \mid 0 \leq(x-x_0) \cdot \nu_0 \leq \delta_0\},
$$
where $\nu_0$ denotes the unit internal normal vector of $\partial \Omega$ at $x_0$. By the smoothness of the boundary $\partial \Omega$, there exists $r_0=r_0(x_0,\Omega)>0$ small enough, such that for any $x \in B_{r_0}(x_0) \cap \partial \Omega$, $u_{\varepsilon}(x)$ is monotone increasing along the internal normal direction at $x$ in the region
$$
\overline{\Sigma_x}:=\{y \in \overline{\Omega} \mid 0 \leq(y-x) \cdot \nu_x \leq \frac{3}{4} \delta_0\},
$$
where $\nu_x$ denotes the unit internal normal vector of $\partial \Omega$ at $x$. Since $x_0 \in \partial \Omega$ is arbitrary and $\partial \Omega$ is compact, we can cover $\partial \Omega$ by finitely many balls $B_{r_k}(x_k)$, $k=0,1,\cdots,N$, where $r_k>0$, $x_k\in \partial \Omega$, and $N$ depends only on $\Omega$. For each $x_k \in \partial \Omega$, choose $\delta_k$ in a similar way as $\delta_0$. Let
$$
\bar{\delta}:=\frac{3}{4} \min \{\delta_0, \delta_1, \cdots, \delta_N\},
$$
then for any $x \in \partial \Omega$, $u(x+t \nu_x)$ is monotone increasing with respect to $t \in[0, \bar{\delta}]$, the proof of \eqref{eq:uvare<CinOmega1} is finished.

Denote
$$
\Omega_2:=\{y \in \Omega \mid \operatorname{dist}(y, \partial \Omega) \geq \bar{\delta}\}.
$$
Then by Proposition \ref{pro:IMRN}, we have
$$
\|u_{\varepsilon}\|_{L^{\infty}(\Omega_2)}\leq C.
$$
Indeed, in the proof of Proposition \ref{pro:IMRN}, we use $B_1$ and $B_3$ for conveniences. One can replace them by open sets. Note that the constant $C$ in Proposition \ref{pro:IMRN} depends on $\inf _{B_2} a$, so we cannot guarantee that the conclusion of this proposition holds for any $\varepsilon\in(0, \lambda_1)$.

Therefore, the proof of this proposition is completed.
\end{proof}

Notice that \eqref{eq:higherorder} has no solution for $\varepsilon\geq \lambda_1$ (see \cite[page 243]{GGSPolyharmonic2010}), therefore, the only solution to
$$
\begin{cases}
(-\Delta)^m u=\lambda_1 u+u^{p} & \text { in }\, \Omega, \\
u\geq0 & \text { in }\, \Omega, \\
u=-\Delta u=\cdots=(-\Delta)^{m-1} u=0 & \text { on }\, \partial \Omega,
\end{cases}
$$
is the trivial one. Using this fact and Proposition \ref{pro:uinfleqC}, we have the following corollary.

\begin{corollary}\label{cor:uinfto0}
Let $u_{\varepsilon}$ be a solution of \eqref{eq:higherorder}. Then $\|u_{\varepsilon}\|_{L^{\infty}(\Omega)} \rightarrow 0$ as $\varepsilon \rightarrow \lambda_{1}$.
\end{corollary}

Now we can give the proof of Theorem \ref{thm:1}.

\begin{proof}[Proof of Theorem \ref{thm:1}]
We prove Theorem \ref{thm:1} by a contradiction argument and suppose on the contrary that there exist sequences $\varepsilon_{i} \rightarrow \lambda_{1}$ as $i \rightarrow \infty$, $\{u_{\varepsilon_{i}}\}$ and $\{v_{\varepsilon_{i}}\}$ such that $u_{\varepsilon_{i}}$ and $v_{\varepsilon_{i}}$ solve \eqref{eq:higherorder} for $\varepsilon=\varepsilon_{i}$ and $\|u_{\varepsilon_{i}}-v_{\varepsilon_{i}}\|_{L^{\infty}(\Omega)} \neq 0$. Denote
$$
w_{\varepsilon_{i}}=\frac{u_{\varepsilon_{i}}-v_{\varepsilon_{i}}}{\|u_{\varepsilon_{i}}-v_{\varepsilon_{i}}\|_{L^{\infty}(\Omega)}},
$$
then $w_{\varepsilon_{i}}$ satisfies
$$
\begin{cases}
(-\Delta)^m w_{\varepsilon_{i}}=\varepsilon_{i} w_{\varepsilon_{i}}+c_{\varepsilon_{i}} w_{\varepsilon_{i}} & \text { in }\, \Omega, \\
w_{\varepsilon_{i}}=-\Delta w_{\varepsilon_{i}}=\cdots=(-\Delta)^{m-1} w_{\varepsilon_{i}}=0 & \text { on }\, \partial \Omega,
\end{cases}
$$
where $c_{\varepsilon_{i}}=p \xi_{\varepsilon_{i}}^{p-1}$ and $\xi_{\varepsilon_{i}}$ is a positive function between $u_{\varepsilon_{i}}$ and $v_{\varepsilon_{i}}$. By Corollary \ref{cor:uinfto0}, we have $\|c_{\varepsilon_{i}}\|_{L^{\infty}(\Omega)} \rightarrow 0$ as $i \rightarrow \infty$. Since $\|w_{\varepsilon_{i}}\|_{L^{\infty}(\Omega)}=1$, by standard elliptic estimates we conclude that up to subsequences, $w_{\varepsilon_{i}} \rightarrow w_{0}$ in $C^{2m}(\overline{\Omega})$, and $w_{0}$ satisfies
$$
\begin{cases}
(-\Delta)^m w_{0}=\lambda_{1} w_{0} & \text { in }\, \Omega, \\
w_{0}=-\Delta w_{0}=\cdots=(-\Delta)^{m-1} w_{0}=0 & \text { on }\, \partial \Omega.
\end{cases}
$$
Therefore, $w_{0}$ is a first Navier eigenfunction of $(-\Delta)^m$, and satisfies
\begin{equation}\label{eq:proofthm1-1}
w_{0}>0\quad \text { in }\, \Omega, \quad \frac{\partial w_{0}}{\partial \nu}<0\quad \text { on }\, \partial \Omega,
\end{equation}
where $\nu$ is the unit outward normal vector of $\partial \Omega$.

On the other hand, since $u_{\varepsilon_{i}}$ and $v_{\varepsilon_{i}}$ satisfy \eqref{eq:higherorder} with $\varepsilon=\varepsilon_{i}$, we have
$$
\int_{\Omega}(-\Delta)^m u_{\varepsilon_{i}} v_{\varepsilon_{i}}\,\ud x=\int_{\Omega} \varepsilon_{i} u_{\varepsilon_{i}} v_{\varepsilon_{i}}\,\ud x+\int_{\Omega} u_{\varepsilon_{i}}^p v_{\varepsilon_{i}}\,\ud x
$$
and
$$
\int_{\Omega}(-\Delta)^m v_{\varepsilon_{i}} u_{\varepsilon_{i}}\,\ud x=\int_{\Omega} \varepsilon_{i} v_{\varepsilon_{i}} u_{\varepsilon_{i}}\,\ud x+\int_{\Omega} v_{\varepsilon_{i}}^p u_{\varepsilon_{i}}\,\ud x.
$$
Using the Navier boundary conditions, we have
$$
\int_{\Omega}(-\Delta)^m u_{\varepsilon_{i}} v_{\varepsilon_{i}}\,\ud x=\int_{\Omega}(-\Delta)^m v_{\varepsilon_{i}} u_{\varepsilon_{i}}\,\ud x.
$$
Hence,
$$
\int_{\Omega} u_{\varepsilon_{i}} v_{\varepsilon_{i}}(u_{\varepsilon_{i}}^{p-1}-v_{\varepsilon_{i}}^{p-1})\,\ud x=0.
$$
Note that $u_{\varepsilon_i}>0$ and $v_{\varepsilon_i}>0$ in $\Omega$. Therefore, for all $i$, there exists $x_{i} \in \Omega$ such that $w_{\varepsilon_{i}}(x_i)=0$. Using \eqref{eq:proofthm1-1} and the fact that $w_{\varepsilon_{i}} \rightarrow w_{0}$ in $C^{2m}(\overline{\Omega})$, we know that $x_i \rightarrow x_{0} \in \partial \Omega$. Let $\bar{x}_i$ be one of the closest points to $x_i$ on $\partial \Omega$. Then the direction of $\overline{x_i \bar{x}_i}$ is close to the exterior unit normal $\nu$ at $x_{0}$ as $i \rightarrow \infty$. By the boundary conditions of $u_{\varepsilon_{i}}$ and $v_{\varepsilon_{i}}$, we have $w_{\varepsilon_{i}}(\bar{x}_i) = 0$, it then follows from the fact $w_{\varepsilon_{i}} \rightarrow w_{0}$ in $C^{2m}(\overline{\Omega})$ that
$$
\frac{\partial w_{0}}{\partial \nu}\Big|_{x=x_{0}}=0,
$$
which contradicts \eqref{eq:proofthm1-1}. Therefore, we complete the proof.
\end{proof}

\section{Asymptotic behavior of solutions to \eqref{eq:Navier}}

In this section, we study the asymptotic behavior of solutions to \eqref{eq:Navier}. The main result of this section is as follows.

\begin{proposition}\label{pro:Han}
Let $\Omega$ be a strictly convex smooth bounded domain in $\mathbb{R}^{n}$ with $n \geq 12$, satisfying $(A)$. Let $u_{\varepsilon}$ be a solution of \eqref{eq:Navier}. Then
\begin{enumerate}[(i)]
  \item
we have
\be\label{eq:Han-1}
u_{\varepsilon}(x) \leq C \frac{\|u_{\varepsilon}\|_{L^{\infty}(\Omega)}}{(1+\|u_{\varepsilon}\|_{L^{\infty}(\Omega)}^{4 /(n-4)}|x|^{2})^{(n-4) / 2}},
\ee
where $C$ is a positive constant independent of $\varepsilon$;
  \item
for any neighborhood $\omega$ of $\partial \Omega$ not containing $x=0$, we have
\be\label{eq:Han-2}
\|u_{\varepsilon}\|_{L^{\infty}(\Omega)} u_{\varepsilon}(x) \rightarrow \frac{2c_0^{-n/2}\omega_{n}}{n(n+2)} G(x, 0) \quad \text { as }\, \varepsilon \rightarrow 0
\ee
and
\be\label{eq:Han-2-1}
\Delta(\|u_{\varepsilon}\|_{L^{\infty}(\Omega)} u_{\varepsilon}(x)) \rightarrow \frac{2c_0^{-n/2}\omega_{n}}{n(n+2)} \Delta_x G(x, 0) \quad \text { as }\, \varepsilon \rightarrow 0
\ee
uniformly in $\omega$, where $c_0$ is defined in \eqref{eq:constantc0} and $G$ is the Green function of $(-\Delta)^2$ under the Navier boundary condition;
  \item
we have
\be\label{eq:Han-3}
\lim _{\varepsilon \rightarrow 0} \varepsilon\|u_{\varepsilon}\|_{L^{\infty}(\Omega)}^{\frac{2(n-8)}{n-4}}=-C_n R(0),
\ee
where $C_n$ is defined in \eqref{eq:constantCn} and $R(x)$ is the Robin function.
\end{enumerate}
\end{proposition}

\begin{remark}
For later applications, we let $\Omega$ satisfy some convexity and symmetry assumptions in Proposition \ref{pro:Han}. In fact, the corresponding conclusions are also true for the general region.
\end{remark}

\subsection{Proof of \eqref{eq:leastenergy}}

As mentioned in the introduction, we first use the results of blow up analysis obtained in \cite{NTZCompactness2022} to prove that \eqref{eq:leastenergy} holds when $\Omega$ satisfies ($A$).

As in the proof of Proposition \ref{pro:uinfleqC}, \eqref{eq:Navier} is equivalent to the elliptic system
\be\label{eq:equiveq}
\begin{cases}
-\Delta u_{1}=u_{2} & \text { in }\, \Omega,\\
-\Delta u_{2}=\varepsilon u_{1}+u_{1}^p & \text { in }\, \Omega,\\
u_1>0,\, u_2>0 & \text { in }\, \Omega, \\
u_{1}=u_{2}=0 & \text { on }\, \partial\Omega.
\end{cases}
\ee
It then follows from assumption $(A)$ and \cite[Lemma 4.3]{TSymmetry1981} that $u_{\varepsilon}$ is symmetric with respect to the hyperplanes $\{x_{i}=0\}$, i.e.,
$$
u_{\varepsilon}(x_{1}, \cdots,-x_{i}, \cdots, x_{n})=u_{\varepsilon}(x_{1}, \cdots, x_{i}, \cdots, x_{n}),\quad i=1, \cdots, n,
$$
and
\be\label{eq:maxu=u0}
\|u_{\varepsilon}\|_{L^{\infty}(\Omega)}=u_{\varepsilon}(0).
\ee

We claim that
\be\label{eq:utoinfty}
\|u_{\varepsilon}\|_{L^{\infty}(\Omega)} \rightarrow+\infty\quad \text{ as }\, \varepsilon \rightarrow 0.
\ee
Indeed, suppose by contradiction that the result is false. Then, there exist a sequence $\varepsilon_{i} \rightarrow 0$ as $i \rightarrow \infty$ such that $\|u_{\varepsilon_{i}}\|_{L^{\infty}(\Omega)} \leq C$ for some constant $C>0$. Therefore, there exists a subsequence of $\varepsilon_{i}$, still denoted by $\varepsilon_{i}$, and a constant $C_0\geq0$ such that, $\|u_{\varepsilon_i}\|_{L^{\infty}(\Omega)}\rightarrow C_{0}$ as $i \rightarrow \infty$. By \eqref{eq:Navier} and integration by parts we find that
$$
\int_{\Omega}|\Delta u_{\varepsilon_{i}}|^{2} \,\ud x=\varepsilon_i\int_{\Omega}|u_{\varepsilon_{i}}|^{2}\,\ud x+\int_{\Omega}|u_{\varepsilon_{i}}|^{p+1} \,\ud x.
$$
Then by Sobolev inequality \eqref{eq:Sobolevine} and H\"older inequality, we have
\be\label{eq:SobolevHolder}
S_{n}  \leq \frac{\int_{\Omega}|\Delta u_{\varepsilon_{i}}|^{2} \,\ud x}{(\int_{\Omega}|u_{\varepsilon_{i}}|^{p+1} \,\ud x)^{2 / (p+1)}} \leq\varepsilon_{i}|\Omega|^{\frac{p-1}{p+1}}+\Big(\int_{\Omega}|u_{\varepsilon_{i}}|^{p+1} \,\ud x\Big)^{\frac{p-1}{p+1}} \leq (\varepsilon_i+\|u_{\varepsilon_{i}}\|^{p-1}_{L^{\infty}(\Omega)}) |\Omega|^{\frac{p-1}{p+1}}.
\ee
This implies that $C_{0}>0$. On the other hand, by standard elliptic theory we conclude that $u_{\varepsilon_{i}}$, up to a subsequence, converges in ${C}^{4}(\overline{\Omega})$ to some function $u_0$ and $u_0$ satisfies
\be\label{eq:eqvare=0}
\begin{cases}
(-\Delta)^2 u=u^{p} & \text { in }\, \Omega, \\
u \geq 0 & \text { in }\, \Omega,\\
u=-\Delta u=0 & \text { on }\, \partial \Omega.
\end{cases}
\ee
Furthermore,
$$
u(0)=\lim_{i\to\infty}\|u_{\varepsilon_{i}}\|_{L^{\infty}(\Omega)}=C_0>0.
$$
Note that assumption $(A)$ means that $\Omega$ is star-shaped with respect to the origin. Then by Van der Vorst \cite{VBest1995}, \eqref{eq:eqvare=0} has only trivial solution. This is a contradiction and we finish the proof of the claim.

Let $\mu_{\varepsilon}>0$ such that
\be\label{eq:defmuvare}
\mu_{\varepsilon}^{-\frac{n-4}{2}}=\|u_{\varepsilon}\|_{L^{\infty}(\Omega)}.
\ee
It follows from \eqref{eq:utoinfty} that
$$
\mu_{\varepsilon} \to 0\quad \text{ as }\, \varepsilon \rightarrow 0.
$$
Denote
\be\label{eq:vvaredef}
v_{\varepsilon}(y):=\mu_{\varepsilon}^{\frac{n-4}{2}} u_{\varepsilon}(\mu_{\varepsilon} y).
\ee
Then $v_{\varepsilon}$ satisfies
\be\label{eq:vvareeq}
\begin{cases}
(-\Delta)^2 v_{\varepsilon} =\varepsilon \mu_{\varepsilon}^4 v_{\varepsilon}+v_{\varepsilon}^{p} & \text { in } \Omega_{\varepsilon}, \\
v_{\varepsilon}>0 & \text { in } \Omega_{\varepsilon}, \\
v_{\varepsilon} =-\Delta v_{\varepsilon}=0 & \text { on } \partial \Omega_{\varepsilon},
\end{cases}
\ee
where $\Omega_{\varepsilon}=\Omega / \mu_{\varepsilon}$. Note that $0 \leq v_{\varepsilon} \leq 1$ in $\Omega_{\varepsilon}$, $v_{\varepsilon}(0)=1$ and $\Omega_{\varepsilon}$ tends to $\mathbb{R}^{n}$ as $\varepsilon \rightarrow 0$. By standard elliptic theory we know that there exists a subsequence of $\{v_{\varepsilon}\}$ (still denoted by $v_{\varepsilon}$) converging to some $V$ uniformly on every compact subset, and $V$ satisfies
\be\label{eq:v}
\begin{cases}
(-\Delta)^2 V =V^p \quad \text { in }\, \mathbb{R}^{n}, \\
V(0) =1, \, V(\infty)=0, \\
0<V \leq 1.
\end{cases}
\ee
It was shown by \cite[Theorem 2.1]{EFJCritical1990} and \cite[Theorem 1.3]{LA1998} that \eqref{eq:v} admits a unique solution given by
\be\label{eq:V(x)def}
V(x)=\Big(\frac{1}{1+c_0|x|^2}\Big)^{\frac{n-4}{2}},
\ee
where
\be\label{eq:constantc0}
c_0=((n-4)(n-2) n(n+2))^{\frac{1}{2}}.
\ee
Moreover, if we finish the proof of \eqref{eq:leastenergy}, we have that
\be\label{eq:globalcompact}
v_{\varepsilon} \rightarrow V \quad \text{ in }\, H^2(\mathbb{R}^n)\quad \text{ as }\, \varepsilon \rightarrow 0.
\ee
This follows from \cite{GGSExistence2003}, see also the arguments in \cite{HAsymptotic1991} and \cite{CGAsymptotics2000}.

Recall the definition of blow up point, isolated blow up point, and isolated simple blow up point given in subsection 2.1. It follows from \eqref{eq:maxu=u0} and \eqref{eq:utoinfty} that $x=0$ is a blow up point. In the following, we will prove that $x=0$ is an isolated blow up point. Furthermore, Lemma \ref{lem:IMRN-isomustisosimle} tells us that $x=0$ is also an isolated simple blow up point.

For convenience, in the following we let $\varepsilon_{i} \rightarrow 0$ as $i \rightarrow \infty$, and let $u_i=u_{\varepsilon_i}$ be a solution of \eqref{eq:Navier} with $\varepsilon=\varepsilon_i$.

\begin{proposition}\label{pro:ui-isolated}
Let $\Omega$ satisfies $(A)$ and let $u_{i}$ be a solution of \eqref{eq:Navier}. Then there exists a positive constant $C$ independent of $i$, such that
$$
u_{i}(x) \leq C|x|^{-(n-4)/2}\quad \text{ in }\, \Omega.
$$
\end{proposition}

\begin{proof}
We are going to prove this proposition by contradiction. Assume that for each $i$,
$$
\max _{x\in \overline{\Omega}}|x|^{\frac{n-4}{2}} u_{i}(x) \geq i.
$$
Pick a point ${x}_i$ in $\overline{\Omega}$ that satisfies
$$
|x_{i}|^{\frac{n-4}{2}} u_{i}(x_{i})=\max _{x\in \overline{\Omega}}|x|^{\frac{n-4}{2}} u_{i}(x).
$$
By \eqref{eq:uvare<CinOmega1} and \eqref{eq:utoinfty}, we have $x_i\in \{y \in \Omega \mid \operatorname{dist}(y, \partial \Omega) \geq \bar{\delta}\}$. Let $\delta\in (0,\bar{\delta})$ such that $B_{\delta}(x_{i}) \subset \Omega$ and let $R>1$ such that $|x_{i}| /4 R<\delta$ for any $i$.

Consider the following rescaled function
$$
v_{i}(x)={u_{i}(x_{i})^{-1}} u_{i}(u_{i}(x_{i})^{-\frac{2}{n-4}}{x}+x_{i}),\quad \forall\, |x|<R_i:=\frac{1}{4R}|x_i|u_i(x_i)^{\frac{2}{n-4}}.
$$
It follows from the definition of $x_i$ that
$$
R_i\geq \frac{1}{4R}i^{\frac{2}{n-4}}\to \infty \quad \text{ as }\, i \rightarrow \infty.
$$
Let $y=u_{i}(x_{i})^{-\frac{2}{n-4}}x+x_{i}$, $x\in B_{R_i}$, then $|y-x_{i}| \leq|x_{i}| /4 R<\delta$ and
$$
|y|\geq|x_i|-|y-x_i|\geq \Big(1-\frac{1}{4R}\Big)|x_i|\geq \frac{|x_i|}{2}.
$$
Hence, we have
$$
|x_{i}|^{\frac{n-4}{2}} u_{i}(x_{i})\geq |y|^{\frac{n-4}{2}} u_{i}(y)\geq \Big(\frac{|x_{i}|}{2}\Big)^{\frac{n-4}{2}} u_{i}(y),\quad \forall\, |x|<R_i.
$$
It follows that
$$
v_{i}(x)=\frac{u_{i}(y)}{u_{i}(x_{i})} \leq 2^{\frac{n-4}{2}} \leq C, \quad \forall\, |x|<R_i.
$$
By the arguments of \eqref{eq:vvaredef}-\eqref{eq:v}, there exists a subsequence of $v_i$ (still denoted by $v_i$) that converges to $v(x)$ uniformly on every compact subset, where $v(x)$ is a positive solution of
$$
\begin{cases}
(-\Delta)^2 v=v^{p} & \text { in } \, \mathbb{R}^{n}, \\
v(0)=1.
\end{cases}
$$
By \cite[Theorem 2.1]{EFJCritical1990} and \cite[Theorem 1.3]{LA1998}, we have
$$
v(x)=c_1\Big(\frac{\lambda}{1+\lambda^{2}|x-x_{0}|^{2}}\Big)^{\frac{n-4}{2}},
$$
where $c_1=c_0^{-(n-4)/4}$, $\lambda>0$, and $x_0\in \mathbb{R}^n$.

Since $1=v(0)=c_1(\frac{\lambda}{1+ \lambda^2|x_0|^2})^{\frac{n-4}{2}}$ and $v(x) \leq 2^{\frac{n-4}{2}}$ for all $x \in \mathbb{R}^n$, we see that
$$
1\leq c_1\lambda^{\frac{n-4}{2}}\leq 2^{\frac{n-4}{2}},
$$
and
$$
|x_{0}|^{2}=\frac{c_1^{\frac{2}{n-4}}\lambda-1}{\lambda^{2}}.
$$
Therefore,
$$
|x_{0}|^{2} \leq {c_1^{\frac{4}{n-4}}}/{4}.
$$
Note that $v_{i} \rightarrow v$ in $C^{1}(B_{\delta_0}(x_{0}))$ for some $\delta_0>0$. Hence, there exists a sequence $z_{i} \in B_{\delta_0}(x_{0})$, such that $\nabla v_{i}(z_{i})=0$. By the assumption $(A)$ of $\Omega$ and the moving plane method, we know that $\nabla u_{i}(y)=0$ if and only if $y=0$. This implies that
$$
u_{i}(x_{i})^{-\frac{2}{n-4}}z_{i}+x_{i}=0.
$$
Hence one obtains
$$
|z_{i}|=|x_{i}| u_{i}(x_{i})^{\frac{2}{n-4}}\geq i^{\frac{2}{n-4}} \rightarrow \infty \quad \text{ as }\, i \rightarrow \infty,
$$
which contradicts the fact that $z_{i} \in B_{\delta_0}(x_{0})$. The proof is complete.
\end{proof}

By Lemma \ref{lem:IMRN-isomustisosimle}, $x=0$ has to be isolated simple blow up point. Then the conclusion of Lemma \ref{lem:IMRN-sharp} holds, i.e.,
\be\label{eq:|x|<delta}
u_i(x) \leq C \|u_{i}\|_{L^{\infty}(\Omega)}^{-1}|x|^{4-n} \quad \text { in }\, \overline{B}_{\delta},
\ee
where $\delta>0$ is some constant. The following proposition gives an estimate of $u_i$ in $\Omega \backslash \overline{B}_{\delta}$.

\begin{proposition}\label{pro:ui<outBdelta}
Let $\Omega$ satisfies $(A)$ and let $u_{i}$ be a solution of \eqref{eq:Navier}. Then there exists a positive constant $C$ independent of $i$, such that
\be\label{eq:|x|>delta}
u_{i}(x) \leq C \|u_{i}\|_{L^{\infty}(\Omega)}^{-1}\delta^{4-n} \quad \text { in }\, \Omega \backslash \overline{B}_{\delta}.
\ee
\end{proposition}

\begin{proof}
We prove \eqref{eq:|x|>delta} by contradiction. Suppose that there exists a point $x_{i} \in \Omega \backslash \overline{B}_{\delta}$, such that for any $i$,
$$
u_{i}(x_{i})>i\|u_{i}\|_{L^{\infty}(\Omega)}^{-1}\delta^{4-n}.
$$
Let $\bar{x}_i$ satisfies
$$
u_i(\bar{x}_i)=\max_{x\in \overline{\Omega}\backslash B_{\delta}} u_i(x).
$$
Obviously, $\bar{x}_i\notin \partial \Omega$. It follows from \eqref{eq:|x|<delta} that $u_i(x) \leq C \|u_{i}\|_{L^{\infty}(\Omega)}^{-1}\delta^{4-n}$ for $x\in \partial B_{\delta}$, therefore, $\bar{x}_{i} \notin \partial B_{\delta}$, because otherwise $u_{i}(x_{i})>u_{i}(\bar{x}_{i})$ for $i$ big enough. Hence, $\bar{x}_{i} \in \Omega \backslash \overline{B}_{\delta}$ and then $\nabla u_{i}(\bar{x}_{i})=0$. By the assumption $(A)$ of $\Omega$ and the moving plane method, we have $\bar{x}_{i}=0$. This is a contradiction.
\end{proof}

Now we can give the proof of \eqref{eq:leastenergy}.

\begin{proposition}\label{pro:leastenergy}
Let $\Omega$ satisfies $(A)$ and let $u_{i}$ be a solution of \eqref{eq:Navier}. Then,
$$
\lim _{i \rightarrow \infty} \frac{\int_{\Omega}|\Delta u_{i}|^2 \,\ud x}{(\int_{\Omega}|u_{i}|^{p+1}\,\ud x)^{2 / (p+1)}}=S_n,
$$
where $S_n$ is given by \eqref{eq:constantSn}.
\end{proposition}

\begin{proof}
By \eqref{eq:SobolevHolder}, we have
\be\label{eq:SobolevHolder-1}
S_{n}  \leq \frac{\int_{\Omega}|\Delta u_i|^{2} \,\ud x}{(\int_{\Omega}|u_i|^{p+1} \,\ud x)^{2 / (p+1)}} \leq\varepsilon_{i}|\Omega|^{\frac{p-1}{p+1}}+\Big(\int_{\Omega}|u_i|^{p+1} \,\ud x\Big)^{\frac{p-1}{p+1}}.
\ee
Now we estimate $\int_{\Omega}|u_i|^{p+1} \,\ud x$. Let $\delta>0$ be the constant in \eqref{eq:|x|<delta} and \eqref{eq:|x|>delta}. Consider the rescaled function
$$
v_{i}(y)={u_{i}(0)^{-1}} u_{i}(u_{i}(0)^{-\frac{2}{n-4}}{y}),\quad \forall\, |y|<R_i:=\delta u_i(0)^{\frac{2}{n-4}}.
$$
Similar to the arguments of \eqref{eq:vvaredef}-\eqref{eq:V(x)def}, we have that
$$
v_{i}(y) \rightarrow V(y)\quad \text { in }\, C_{l o c}^{4}(\mathbb{R}^{n}).
$$
By the dominate convergence theorem and by direct computation, we have
$$
\int_{B_{\delta}} u_{i}^{p+1}(x) \,\ud x=\int_{B_{R_i}} v_{i}^{p+1}(y) \,\ud y \to \int_{\mathbb{R}^n} V^{p+1}(y) \,\ud y=S_n^{{n}/{4}}\quad \text{ as }\, i\to\infty.
$$
On the other hands, by \eqref{eq:|x|>delta}, we have
$$
\int_{\Omega\backslash B_{\delta}} u_{i}^{p+1}(x) \,\ud x\to 0\quad \text{ as }\, i\to\infty.
$$
Therefore,
$$
\int_{\Omega} u_{i}^{p+1}(x) \,\ud x=\int_{B_{\delta}} u_{i}^{p+1}(x) \,\ud x+\int_{\Omega\backslash B_{\delta}} u_{i}^{p+1}(x) \,\ud x\to S_n^{{n}/{4}}\quad \text{ as }\, i\to\infty.
$$
This and \eqref{eq:SobolevHolder-1} finish the proof.
\end{proof}

\subsection{Preliminary estimates}
Before proving Proposition \ref{pro:Han}, in this subsection we first give some preliminary estimates.

We introduce
\be\label{eq:Kelvin}
w_{\varepsilon}(z):=|z|^{4-n} v_{\varepsilon}\Big(\frac{z}{|z|^{2}}\Big)
\ee
to be the Kelvin transformation of $v_{\varepsilon}$, where $v_{\varepsilon}$ is given in \eqref{eq:vvaredef}. Let $y:=\frac{z}{|z|^{2}} \in \Omega_{\varepsilon}$, then $z \in \Omega_{\varepsilon}^{*}:=\{\frac{y}{|y|^{2}} \mid y \in \Omega_{\varepsilon}\}$. By direct computation (see, for example, \cite[Lemma 3.6]{XUniqueness2000}),
\be\label{eq:wvare-1}
\begin{cases}
-\Delta_{z} w_{\varepsilon}=2(n-4)|z|^{2-n} v_{\varepsilon}(y)+4|z|^{-n}(z \cdot \nabla_{y} v_{\varepsilon}(y))-|z|^{-n} \Delta_{y} v_{\varepsilon}(y),\\
(-\Delta_{z})^{2} w_{\varepsilon}=|z|^{-4-n} (-\Delta_{y})^{2} v_{\varepsilon}(y).
\end{cases}
\ee
By \eqref{eq:Kelvin}, we have
$$
z \cdot \nabla_{y} v_{\varepsilon}(y)=(4-n)|z|^{n-2} w_{\varepsilon}-|z|^{n-2}(z \cdot \nabla_{z} w_{\varepsilon}(z)).
$$
Notice that $w_{\varepsilon} = 0$ on $\partial \Omega_{\varepsilon}^{*}$. Hence,
$$
\nabla w_{\varepsilon}=(\nu \cdot \nabla w_{\varepsilon}) \nu=\frac{\partial w_{\varepsilon}}{\partial \nu} \nu\quad \text{ on }\, \partial \Omega_{\varepsilon}^{*},
$$
where $\nu$ is the unit outward normal vector of $\partial \Omega_{\varepsilon}^{*}$. Since $z \in \partial \Omega_{\varepsilon}^{*}$ if and only if $y \in \partial \Omega_{\varepsilon}$, using the above identities and \eqref{eq:vvareeq}, we have
$$
-\Delta_{z} w_{\varepsilon}=-4|z|^{-2}(z \cdot \nu) \frac{\partial w_{\varepsilon}}{\partial \nu}\quad \text{ on }\, \partial \Omega_{\varepsilon}^{*}.
$$
Therefore, it follows from \eqref{eq:wvare-1} and \eqref{eq:vvareeq} that $w_{\varepsilon}$ satisfies
\be\label{eq:wvare}
\begin{cases}
(-\Delta)^{2} w_{\varepsilon} =\varepsilon \mu_{\varepsilon}^4|z|^{-8}w_{\varepsilon}+w_{\varepsilon}^p & \text { in }\, \Omega_{\varepsilon}^{*}, \\
w_{\varepsilon}>0 & \text { in }\, \Omega_{\varepsilon}^{*}, \\
\displaystyle w_{\varepsilon} =0,\, -\Delta w_{\varepsilon} =-4|z|^{-2}(z \cdot \nu) \frac{\partial w_{\varepsilon}}{\partial \nu} & \text { on }\, \partial \Omega_{\varepsilon}^{*}.
\end{cases}
\ee

Notice that \eqref{eq:Han-1} is equivalent to
\be\label{eq:vvare<CV}
v_{\varepsilon}(y) \leq C V(y), \quad \forall\, y\in \Omega_{\varepsilon},
\ee
or
\be\label{eq:wvare<C}
w_{\varepsilon}(z) \leq C, \quad \forall\, z\in \Omega_{\varepsilon}^*,
\ee
for some $C>0$. For any fixed $R>0$, using \eqref{eq:Kelvin} and the fact that $0 \leq v_{\varepsilon} \leq 1$, we have
\be\label{eq:wvare<R}
w_{\varepsilon}(z) \leq R^{4-n}\quad \text{ for } \, |z| \geq R.
\ee
Hence, to prove \eqref{eq:wvare<C}, we only need to extimate $w_{\varepsilon}$ near the origin as $\varepsilon \rightarrow 0$. More precisely, we only need to prove the following proposition.

\begin{proposition}\label{pro:wvareLinftyBR<C}
Let $\Omega$ satisfies $(A)$ and $w_{\varepsilon}$ is given in \eqref{eq:Kelvin}. Then for any fixed $R>0$, there exists a positive constant $C$ independent of $\varepsilon$, such that
$$
\|w_{\varepsilon}\|_{L^{\infty}(\Omega_{\varepsilon}^{*} \cap B_{R/2})} \leq C.
$$
\end{proposition}

Before proving Proposition \ref{pro:wvareLinftyBR<C}, we first give the following estimate, which plays an important role in the proof of Proposition \ref{pro:wvareLinftyBR<C}.

\begin{lemma}\label{lem:vare-muvare}
There exist a constant $C>0$ such that
$$
\varepsilon\leq C\mu_{\varepsilon}^{n-8}
$$
for all $\varepsilon>0$.
\end{lemma}

\begin{proof}
We will prove this lemma by using the Pohozaev identity \eqref{eq:Pohozaev}. First of all, we shall find an lower bound of the left-hand side of \eqref{eq:Pohozaev}. By \eqref{eq:vvaredef} we have
\be\label{eq:Pohozaevleft}
\begin{aligned}
\varepsilon \int_{\Omega} u_{\varepsilon}^{2}(x) \,\ud x=&\varepsilon \int_{\Omega} \mu_{\varepsilon}^{4-n}v_{\varepsilon}^2(\mu_{\varepsilon}^{-1}x) \,\ud x\\
=&\varepsilon\mu_{\varepsilon}^{4} \int_{\Omega_{\varepsilon}} v_{\varepsilon}^2(y) \,\ud y\\
\geq& \varepsilon\mu_{\varepsilon}^{4} \int_{B_1(0)} v_{\varepsilon}^2(y) \,\ud y\\
\geq& C\varepsilon\mu_{\varepsilon}^{4},
\end{aligned}
\ee
where we used the fact that $v_{\varepsilon}$ converges to $V$ uniformly on any compact set.

In order to estimate the right-hand side of \eqref{eq:Pohozaev}, we need to get the estimates of $\|\nabla u_{\varepsilon}\|_{L^{\infty}(\partial \Omega)}$ and $\|\nabla \Delta u_{\varepsilon}\|_{L^{\infty}(\partial \Omega)}$. We use Lemma \ref{lem:esti-unablau} to estimate them. Since \eqref{eq:Navier} is equivalent to the elliptic system \eqref{eq:equiveq}, by Lemma \ref{lem:esti-unablau}, we have
$$
\|\Delta u_{\varepsilon}\|_{W^{1, q}(\Omega)}+\|\Delta u_{\varepsilon}\|_{C^{1, \alpha}(\omega^{\prime})} \leq C(\|\varepsilon u_{\varepsilon}+ u_{\varepsilon}^{p}\|_{L^{1}(\Omega)}+\|\varepsilon u_{\varepsilon}+ u_{\varepsilon}^{p}\|_{L^{\infty}(\omega)})
$$
and
$$
\|u_{\varepsilon}\|_{W^{1, q}(\Omega)}+\|u_{\varepsilon}\|_{C^{1, \alpha}(\omega^{\prime\prime})} \leq C(\|\Delta u_{\varepsilon}\|_{L^{1}(\Omega)}+\|\Delta u_{\varepsilon}\|_{L^{\infty}(\omega^{\prime})}),
$$
where $q\in [1, \frac{n}{n-1})$, $\alpha \in(0,1)$, $\omega$ is a neighborhood of $\partial \Omega$, and $\omega^{\prime} \subset \omega$ is a strict subdomain of $\omega$, $\omega^{\prime\prime} \subset \omega^{\prime}$ is a strict subdomain of $\omega^{\prime}$.

Using \eqref{eq:vvaredef}, we deduce that
$$
\begin{aligned}
\int_{\Omega} u_{\varepsilon}^{p}(x) \,\ud x =&\mu_{\varepsilon}^{-\frac{p(n-4)}{2}}\int_{\Omega}v_{\varepsilon}^p(\mu_{\varepsilon}^{-1}x)\,\ud x \\
=&\mu_{\varepsilon}^{-\frac{p(n-4)}{2}+n}\int_{\Omega_{\varepsilon}}v_{\varepsilon}^p(y)\,\ud y\\
\leq&\mu_{\varepsilon}^{-\frac{p(n-4)}{2}+n}\Big(C\int_{\Omega_{\varepsilon}}|v_{\varepsilon}(y)-V(y)|^p\,\ud y+C\int_{\Omega_{\varepsilon}}V^p(y)\,\ud y\Big)\\
\leq&C\mu_{\varepsilon}^{\frac{n-4}{2}},
\end{aligned}
$$
where we used the fact that $\lim_{\varepsilon\to 0}\int_{\Omega_{\varepsilon}}|v_{\varepsilon}(y)-V(y)|^p\,\ud y=0$ (see \eqref{eq:globalcompact}) and $\int_{\mathbb{R}^n} V^p(y)\,\ud y\leq C$. Moreover, we have
$$
\begin{aligned}
\int_{\Omega} u_{\varepsilon}(x) \,\ud x=&\mu_{\varepsilon}^{-\frac{n-4}{2}}\int_{\Omega}v_{\varepsilon}(\mu_{\varepsilon}^{-1}x)\,\ud x\\
=&\mu_{\varepsilon}^{\frac{n+4}{2}}\int_{\Omega_{\varepsilon}}v_{\varepsilon}(y)\,\ud y\\
\leq&\mu_{\varepsilon}^{\frac{n+4}{2}}\Big(C\int_{\Omega_{\varepsilon}}|v_{\varepsilon}(y)-V(y)|\,\ud y+C\int_{\Omega_{\varepsilon}}V(y)\,\ud y\Big)\\
\leq&\mu_{\varepsilon}^{\frac{n+4}{2}}\Big(C+C\int_{0}^{C\mu_{\varepsilon}^{-1}}\Big(\frac{1}{1+r^2}\Big)^{\frac{n-4}{2}}r^{n-1}\,\ud r\Big)\\
\leq&C\mu_{\varepsilon}^{\frac{n-4}{2}}.
\end{aligned}
$$
Therefore,
$$
\|\varepsilon u_{\varepsilon}+ u_{\varepsilon}^{p}\|_{L^{1}(\Omega)}\leq \|\varepsilon u_{\varepsilon}\|_{L^{1}(\Omega)}+\|u_{\varepsilon}^{p}\|_{L^{1}(\Omega)}\leq C\mu_{\varepsilon}^{\frac{n-4}{2}}.
$$
By Proposition \ref{pro:ui<outBdelta}, we find that
$$
\|\varepsilon u_{\varepsilon}+ u_{\varepsilon}^{p}\|_{L^{\infty}(\omega)}\leq C\mu_{\varepsilon}^{\frac{n-4}{2}}.
$$
Hence,
\be\label{eq:DeltauW1qC1alpha}
\|\Delta u_{\varepsilon}\|_{W^{1, q}(\Omega)}+\|\Delta u_{\varepsilon}\|_{C^{1, \alpha}(\omega^{\prime})}\leq C\mu_{\varepsilon}^{\frac{n-4}{2}}.
\ee

On the other hand, since
$$
\|\Delta u_{\varepsilon}\|_{L^1(\Omega)}\leq C\|\Delta u_{\varepsilon}\|_{L^q(\Omega)}\leq C\|\Delta u_{\varepsilon}\|_{W^{1, q}(\Omega)}\leq C\mu_{\varepsilon}^{\frac{n-4}{2}}
$$
and
$$
\|\Delta u_{\varepsilon}\|_{L^{\infty}(\omega^{\prime})}\leq\|\Delta u_{\varepsilon}\|_{C^{1, \alpha}(\omega^{\prime})} \leq C\mu_{\varepsilon}^{\frac{n-4}{2}},
$$
we have
\be\label{eq:uvareC1alpha}
\|u_{\varepsilon}\|_{C^{1, \alpha}(\omega^{\prime\prime})} \leq C(\|\Delta u_{\varepsilon}\|_{L^{1}(\Omega)}+\|\Delta u_{\varepsilon}\|_{L^{\infty}(\omega^{\prime})}) \leq C\mu_{\varepsilon}^{\frac{n-4}{2}}.
\ee

Combining \eqref{eq:DeltauW1qC1alpha} and \eqref{eq:uvareC1alpha} gives the bound
$$
\int_{\partial \Omega} \frac{\partial u_{\varepsilon}}{\partial \nu}\frac{\partial (-\Delta u_{\varepsilon})}{\partial \nu}(x \cdot \nu) \,\ud s\leq C\mu_{\varepsilon}^{n-4}.
$$
We put this bound and \eqref{eq:Pohozaevleft} into the Pohozaev identity \eqref{eq:Pohozaev}. Then we finally get
$$
\varepsilon\leq C\mu_{\varepsilon}^{n-8},
$$
which is the desired inequality.
\end{proof}

Now we can give the proof of Proposition \ref{pro:wvareLinftyBR<C}. The proof is very long, so we devide it into several steps.

\begin{proof}[Proof of Proposition \ref{pro:wvareLinftyBR<C}]
{\bf Step 1.} Let $R>0$ be a fixed constant. Let $w_1$ be a solution of the following problem
\be\label{eq:w-1}
\begin{cases}
(-\Delta)^{2} w =0 & \text { in }\, \Omega_{\varepsilon}^{*} \cap B_{R}, \\
w=-\Delta w =0 & \text { on }\, \partial \Omega_{\varepsilon}^{*}, \\
w=w_{\varepsilon},\, -\Delta w =-\Delta w_{\varepsilon} & \text { on }\, \partial B_{R}.
\end{cases}
\ee
Using the convergence of $v_{\varepsilon}$ to $V$ near $\partial B_{1 / R}$ and the definition of $w_{\varepsilon}$, we have
$$
|w_{\varepsilon}| \leq C(R) \quad \text { on }\, \partial B_{R},
$$
where $C(R)$ is a positive constant independent of $\varepsilon$. From the convexity assumption of $\Omega$ we obtain that $z \cdot \nu \leq 0$ on $\partial \Omega_{\varepsilon}^{*}$. On the other hand, since $\frac{\partial w_{\varepsilon}}{\partial \nu} \leq 0$ on $\partial \Omega_{\varepsilon}^{*}$, by \eqref{eq:wvare}, we have
$$
-\Delta w_{\varepsilon}=-4|z|^{-2}(z \cdot \nu) \frac{\partial w_{\varepsilon}}{\partial \nu} \leq 0 \quad \text { on }\, \partial \Omega_{\varepsilon}^{*}.
$$
Then using the maximum principle twice we have
\be\label{eq:w0nabw0<C}
|w_{1}| \leq C(R) \quad \text { in }\, \Omega_{\varepsilon}^{*} \cap B_{R}.
\ee

Similarly, let $w_2$ be a solution of the following problem
\be\label{eq:w-2}
\begin{cases}
(-\Delta)^{2} w= 0 & \text { in }\, \Omega_{\varepsilon}^{*} \cap B_{R}, \\
\displaystyle w=0, \, -\Delta w=-4|z|^{-2}(z \cdot \nu) \frac{\partial w_{\varepsilon}}{\partial \nu}  & \text { on }\, \partial \Omega_{\varepsilon}^{*},\\
w=-\Delta w=0& \text { on }\, \partial B_{R}.
\end{cases}
\ee
Arguing as in the above, we obtain $-4|z|^{-2}(z \cdot \nu) \frac{\partial w_{\varepsilon}}{\partial \nu} \leq 0$ on $\partial \Omega_{\varepsilon}^{*}$. Using the maximum principle twice we get
$$
-\Delta w_2 \leq 0 \quad \text { in }\, \Omega_{\varepsilon}^{*} \cap B_{R}
$$
and
\be\label{eq:w1<0}
w_2 \leq 0 \quad \text { in }\, \Omega_{\varepsilon}^{*} \cap B_{R}.
\ee

Let
\be\label{eq:tildewdef}
\tilde{w}:=w_{\varepsilon}-w_1-w_2
\ee
and
$$
\Omega_{\varepsilon}^R:=\Omega_{\varepsilon}^{*} \cap B_{R}.
$$
It follows from \eqref{eq:wvare}, \eqref{eq:w-1}, and \eqref{eq:w-2} that $\tilde{w}$ satisfies
\be\label{eq:tildew-1}
\begin{cases}
(-\Delta)^{2} \tilde{w} =a(z) w_{\varepsilon} &  \text { in }\, \Omega_{\varepsilon}^R, \\
\tilde{w} =-\Delta \tilde{w}=0  & \text { on }\, \partial \Omega_{\varepsilon}^R,
\end{cases}
\ee
where $a(z):=\varepsilon \mu_{\varepsilon}^4|z|^{-8}+w_{\varepsilon}^{p-1}$. Clearly by the maximum principle,
\be\label{eq:tildew>0}
\tilde{w} \geq 0 \quad \text { in }\, \Omega_{\varepsilon}^R.
\ee
Define
$$
Q(z):=
\begin{cases}
a(z) & \text { in }\, B_{r},\\
\displaystyle\frac{1}{M} a(z) & \text { in }\, B_{R} \backslash \overline{B}_{r},
\end{cases}
$$
where $r \in(0, R)$ and $M>1$ both independent of $\varepsilon$ and will be determined later. Then we can write
\be\label{eq:awvare-1}
a(z) w_{\varepsilon}=Q(z) w_{\varepsilon}+f(z),
\ee
where
$$
f(z)=
\begin{cases}
0 & \text { in }\, B_{r} , \\
\displaystyle(1-\frac{1}{M}) (\varepsilon \mu_{\varepsilon}^4|z|^{-8}w_{\varepsilon}+w_{\varepsilon}^{p}) & \text { in }\, B_{R} \backslash \overline{B}_{r}.
\end{cases}
$$
It is clear that $f \in L^{\infty}(\Omega_{\varepsilon}^R)$. Indeed, it follows from \eqref{eq:wvare<R} that
\be\label{eq:fLinfty}
\|f\|_{L^{\infty}(\Omega_{\varepsilon}^R)} \leq(1-\frac{1}{M}) (\varepsilon \mu_{\varepsilon}^4+1)r^{-(n+4)}.
\ee

Next, we give the estimate of $Q(z)$ in $L^{n / 4}(\Omega_{\varepsilon}^R)$, which will be used later. By the definition of $Q(z)$, we have
\be\label{eq:Qn/4-1}
\begin{aligned}
&\int_{\Omega_{\varepsilon}^R} Q(z)^{\frac{n}{4}} \,\ud z\\
\leq&\int_{\Omega_{\varepsilon}^{*} \cap B_{r}}  (\varepsilon \mu_{\varepsilon}^4|z|^{-8}+w_{\varepsilon}^{p-1})^{\frac{n}{4}} \,\ud z+\frac{1}{M^{\frac{n}{4}}}\int_{B_{R} \backslash B_{r}}(\varepsilon \mu_{\varepsilon}^4|z|^{-8}+w_{\varepsilon}^{p-1})^{\frac{n}{4}} \,\ud z\\
\leq&C\int_{\Omega_{\varepsilon}^{*} \cap B_{R}} (\varepsilon \mu_{\varepsilon}^4|z|^{-8})^{\frac{n}{4}} \,\ud z+C\int_{\Omega_{\varepsilon}^{*} \cap B_{r}}w_{\varepsilon}^{p+1}\,\ud z+\frac{C}{M^{\frac{n}{4}}} \int_{B_{R} \backslash B_{r}}w_{\varepsilon}^{p+1}\,\ud z.
\end{aligned}
\ee
Since $\Omega$ is a bounded domain and $\Omega_{\varepsilon}=\Omega / \mu_{\varepsilon}$, we can assume that $\Omega_{\varepsilon}\subset B_{C^{-1} \mu_{\varepsilon}^{-1}}$, where $C>0$ is a constant. This means $\Omega_{\varepsilon}^*\subset \mathbb{R}^n\backslash B_{C \mu_{\varepsilon}}$. Therefore, we have
\be\label{eq:Qn/4-2}
\int_{\Omega_{\varepsilon}^{*} \cap B_{R}} (\varepsilon \mu_{\varepsilon}^4|z|^{-8})^{\frac{n}{4}} \,\ud z= \varepsilon^{\frac{n}{4}}\mu_{\varepsilon}^n \int_{\Omega_{\varepsilon}^{*} \cap B_{R}} |z|^{-2n} \,\ud z\leq C\varepsilon^{\frac{n}{4}}.
\ee
By direct computations, there holds
$$
\int_{\Omega_{\varepsilon}} |v_{\varepsilon}(x)-V(x)|^{p+1}\,\ud x=\int_{\Omega_{\varepsilon}^*} |w_{\varepsilon}(z)-V(z)|^{p+1}\,\ud z.
$$
Notice that $v_{\varepsilon} \rightarrow V$ in $L^{p+1}$, we have
$$
\int_{\Omega_{\varepsilon}^*} |w_{\varepsilon}(z)-V(z)|^{p+1}\,\ud z\to0 \quad \text{ as }\, \varepsilon\to 0.
$$
Thus for any $c>0$, we can choose $r$ small enough and $M$ large enough such that for all small $\varepsilon$,
\be\label{eq:Qn/4-3}
\int_{\Omega_{\varepsilon}^{*} \cap B_{r}}w_{\varepsilon}^{p+1}\,\ud z+\frac{1}{M^{\frac{n}{4}}} \int_{B_{R} \backslash B_{r}}w_{\varepsilon}^{p+1}\,\ud z\leq c.
\ee
Combining \eqref{eq:Qn/4-1}-\eqref{eq:Qn/4-3} we conclude that for any $\lambda>0$, by choosing the appropriate $r$ and $M$, we have
\be\label{eq:Qn/4-4}
\int_{\Omega_{\varepsilon}^R} Q(z)^{\frac{n}{4}} \,\ud z\leq\lambda
\ee
for $\varepsilon$ sufficient small.

{\bf Step 2.} Define
\be\label{eq:defeta}
\eta(z):=\chi_{\{x \in \Omega_{\varepsilon}^R \mid w_{\varepsilon}(x) \leq N(w_{\varepsilon}(x)-w_1(x))\}}(z)\quad \text{ for }\, z\in \Omega_{\varepsilon}^R,
\ee
where $\chi_{D}$ denotes the characteristic function of $D$, $N>1$, $w_1$ is the solution of \eqref{eq:w-1}. Then by \eqref{eq:tildewdef} and \eqref{eq:w1<0}, we can rewrite \eqref{eq:awvare-1} as
\be\label{eq:awvare<}
\begin{aligned}
a(z) w_{\varepsilon}=&\eta(z) Q(z) w_{\varepsilon}+\tilde{f}(z)\\
\leq&N \eta(z) Q(z)(w_{\varepsilon}(z)-w_1(z))+\tilde{f}(z)\\
\leq&N \eta(z) Q(z) \tilde{w}(z)+\tilde{f}(z),
\end{aligned}
\ee
where $\tilde{f}(z):=(1-\eta(z)) Q(z) w_{\varepsilon}+f(z)$. By \eqref{eq:defeta}, we have
$$
1-\eta(z)=\chi_{\{x \in \Omega_{\varepsilon}^R \mid w_{\varepsilon}(x) <N w_1(x) /(N-1)\}}(z).
$$
We claim that there exists a constant $C>0$ independent of $\varepsilon$, such that
\be\label{eq:tildefinfty}
\|\tilde{f}\|_{L^{\infty}(\Omega_{\varepsilon}^R)} \leq C.
\ee
Indeed, by the definition of $Q(z)$, \eqref{eq:w0nabw0<C}, and \eqref{eq:fLinfty}, we only need to prove that $\|\varepsilon\mu_{\varepsilon}^4|z|^{-8}\|_{L^{\infty}(\Omega_{\varepsilon}^R)}\leq C$. By Lemma \ref{lem:vare-muvare}, we have
\be\label{eq:dimention}
\begin{aligned}
\|\varepsilon \mu_{\varepsilon}^4|z|^{-8}\|_{L^{\infty}(\Omega_{\varepsilon}^R)}\leq& \varepsilon \mu_{\varepsilon}^4 (C\mu_{\varepsilon})^{-8}\\
\leq&C\varepsilon \mu_{\varepsilon}^{-4}\\
\leq& C\mu_{\varepsilon}^{n-12}\\
\leq& C,
\end{aligned}
\ee
where we used the fact that $n\geq 12$, from which the claim follows.

{\bf Step 3.} By \eqref{eq:awvare<}, we can write \eqref{eq:tildew-1} in the form
\be\label{eq:tildew-2}
\begin{cases}
(-\Delta)^{2} \tilde{w}  \leq N \eta(z) Q(z) \tilde{w}+\tilde{f}  & \text { in }\, \Omega_{\varepsilon}^R, \\
\tilde{w} =-\Delta \tilde{w}=0  & \text { on }\, \partial \Omega_{\varepsilon}^R.
\end{cases}
\ee
In this step we will prove that $\tilde{w}$ is uniformly bounded in $L^{q}(\Omega_{\varepsilon}^R)$ for any $q>1$.

Let $G(x, y)$ denote the Green function of $(-\Delta)^{2}$ on $\Omega_{\varepsilon}^R$ under the Navier boundary condition:
$$
\begin{cases}
(-\Delta)^{2} G(\cdot, y) =\delta_{y}(\cdot)  & \text { in }\, \Omega_{\varepsilon}^R, \\
G(\cdot, y) =-\Delta G(\cdot, y)=0  & \text { on }\, \partial \Omega_{\varepsilon}^R.
\end{cases}
$$
Let $G(x, y)=\Gamma(x, y)+g(x, y)$, where $\Gamma(x, y)$ is the fundamental solution of $(-\Delta)^{2}$ on $\mathbb{R}^n$, see \eqref{eq:fundsolu}, and $g(x, y)$ the regular part of $G$, see \eqref{eq:regupart}. By the maximum principle, we have that
\be\label{eq:G>=0}
G(x, y) \geq 0 \quad \text { in }\, \Omega_{\varepsilon}^R \times \Omega_{\varepsilon}^R.
\ee
Let $q>1$ and $v \in L^{q}(\Omega_{\varepsilon}^R)$, define
$$
\Delta^{-2} v(y):=\int_{\Omega_{\varepsilon}^R} G(x, y) v(x) \,\ud x,\quad y\in \Omega_{\varepsilon}^R.
$$
Now we claim that $\Delta^{-2}$ is bounded, independently of $\varepsilon$, from $L^{q}(\Omega_{\varepsilon}^R)$ to $L^{q}(\Omega_{\varepsilon}^R)$. Let $\frac{1}{q^{\prime}}=\frac{1}{q}+\frac{4}{n}$. Firstly, \eqref{eq:G>=0} yields
\be\label{eq:Delta-2v}
\|\Delta^{-2} v\|_{L^{q}(\Omega_{\varepsilon}^R)} \leq 2\Big\|\int_{\Omega_{\varepsilon}^R} \Gamma(x, y) v(x) \,\ud x\Big\|_{L^{q}(\Omega_{\varepsilon}^R)}.
\ee
Therefore, using \eqref{eq:fundsolu}, the Hardy-Littlewood-Sobolev's inequality (see \cite{LSharp1983}), and the H\"older inequality, we get
\begin{equation}\label{4.31}
\begin{aligned}
\Big\|\int_{\Omega_{\varepsilon}^R} \Gamma(x, y) v(x) \,\ud x\Big\|_{L^{q}(\Omega_{\varepsilon}^R)} \leq& C\|v\|_{L^{q^{\prime}}(\Omega_{\varepsilon}^R)}\\
\leq& C|\Omega_{\varepsilon}^R|^{{4}/{n}}\|v\|_{L^{q}(\Omega_{\varepsilon}^R)}\\
\leq& C\|v\|_{L^{q}(\Omega_{\varepsilon}^R)},
\end{aligned}
\end{equation}
where $C>0$ is some constant independent of $\varepsilon$. Hence our claim is valid.

Note that $G(x,y)$ is non-negative, so the operator $\Delta^{-2}$ is non-negative, then by \eqref{eq:tildew-2} we can write
$$
(I-N \Delta^{-2}(\eta Q)) \tilde{w} \leq \Delta^{-2} \tilde{f}.
$$
According to \eqref{eq:Delta-2v} and \eqref{4.31}, we have for $v \in L^{q}(\Omega_{\varepsilon}^R)$,
$$
\begin{aligned}
\|N\Delta^{-2}(\eta Q) v\|_{L^{q}(\Omega_{\varepsilon}^R)} & \leq C\|(\eta Q) v\|_{L^{q^{\prime}}(\Omega_{\varepsilon}^R)} \\
& \leq C\|\eta Q\|_{L^{{n}/{4}}(\Omega_{\varepsilon}^R)}\|v\|_{L^{q}(\Omega_{\varepsilon}^R)}.
\end{aligned}
$$
Choosing $\lambda$ in \eqref{eq:Qn/4-4} small enough, such that $C\|\eta Q\|_{L^{{n}/{4}}(\Omega_{\varepsilon}^R)}\leq 1/2$. Therefore, $I-N \Delta^{-2}(\eta Q)$ is invertible, and
$$
\tilde{w} \leq(I-N \Delta^{-2}(\eta Q))^{-1} \Delta^{-2} \tilde{f}.
$$
By \eqref{eq:tildew>0} and \eqref{eq:tildefinfty} we have that
\be\label{eq:tildew<C}
\begin{aligned}
\|\tilde{w}\|_{L^{q}(\Omega_{\varepsilon}^R)} & \leq \|(I-N \Delta^{-2}(\eta Q))^{-1}\|\|\Delta^{-2}\|\|\tilde{f}\|_{L^{q}(\Omega_{\varepsilon}^R)}\\
& \leq C\|\tilde{f}\|_{L^{q}(\Omega_{\varepsilon}^R)} \\
& \leq C\|\tilde{f}\|_{L^{\infty}(\Omega_{\varepsilon}^R)} \\
&\leq C,
\end{aligned}
\ee
where $C>0$ independent of $\varepsilon$.

{\bf Step 4.} Now we can finish the proof of this proposition. First of all, by \eqref{eq:tildewdef}, \eqref{eq:w1<0}, and \eqref{eq:tildew>0} we have
$$
\tilde{w} \geq(w_{\varepsilon}-w_1)^{+} .
$$
Therefore,
$$
\begin{aligned}
\|w_{\varepsilon}\|_{L^{q}(\Omega_{\varepsilon}^R)} & \leq\|(w_{\varepsilon}-w_1)^{+}\|_{L^{q}(\Omega_{\varepsilon}^R)}+\|(w_{\varepsilon}-w_1)^{-}\|_{L^{q}(\Omega_{\varepsilon}^R)}
+\|w_1\|_{L^{q}(\Omega_{\varepsilon}^R)} \\
&\leq\|\tilde{w}\|_{L^{q}(\Omega_{\varepsilon}^R)}+\|w_1-w_{\varepsilon}\|_{L^{q}(\{x\in \Omega_{\varepsilon}^R \mid w_{\varepsilon}(x)\leq w_1(x)\})}+\|w_1\|_{L^{q}(\Omega_{\varepsilon}^R)}\\
&\leq\|\tilde{w}\|_{L^{q}(\Omega_{\varepsilon}^R)}+3\|w_1\|_{L^{q}(\Omega_{\varepsilon}^R)},
\end{aligned}
$$
then by \eqref{eq:tildew<C} and \eqref{eq:w0nabw0<C}, we obtain
\be\label{eq:wvareLq<C}
\|w_{\varepsilon}\|_{L^{q}(\Omega_{\varepsilon}^R)} \leq C,
\ee
where $C>0$ independent of $\varepsilon$.

Rewrite \eqref{eq:wvare} as
$$
(-\Delta)^{2} w_{\varepsilon} =a(z)w_{\varepsilon} \quad \text { in }\, \Omega_{\varepsilon}^{*},
$$
where
$$
a(z)=\varepsilon \mu_{\varepsilon}^4|z|^{-8}+w_{\varepsilon}^{p-1}.
$$
We will use Lemma \ref{lem:localbound} to prove that $w_{\varepsilon}$ is bounded on $\Omega_{\varepsilon}^{*} \cap B_{R/2}$. Using \eqref{eq:wvareLq<C}, we only need to show that $a(z)\in L^{{n}/{4}+\tau}(\Omega_{\varepsilon}^R)$ for some $\tau>0$. It follows from Lemma \ref{lem:vare-muvare} that we can choose $\tau>0$ such that
$$
\begin{aligned}
\|\varepsilon \mu_{\varepsilon}^4|z|^{-8}\|_{L^{\frac{n}{4}+\tau}(\Omega_{\varepsilon}^R)}\leq& \varepsilon \mu_{\varepsilon}^4 \Big(\int_{C\mu_{\varepsilon}}^{R}r^{-n-8\tau-1} \,\ud r \Big)^{\frac{4}{n+4\tau}}\\
\leq&C\varepsilon \mu_{\varepsilon}^{-\frac{16\tau}{n+4\tau}}\\
\leq& C\mu_{\varepsilon}^{n-8-\frac{16\tau}{n+4\tau}}\\
\leq& C.
\end{aligned}
$$
On the other hand, by \eqref{eq:wvareLq<C}, we have
$$
\|w_{\varepsilon}^{p-1}\|_{L^{\frac{n}{4}+\tau}(\Omega_{\varepsilon}^R)}=
\Big(\int_{\Omega_{\varepsilon}^R}w_{\varepsilon}^{\frac{2(n+4\tau)}{n-4}}(z)\,\ud z\Big)^{\frac{4}{n+4\tau}}\leq C.
$$
This completes the proof of Proposition \ref{pro:wvareLinftyBR<C}.
\end{proof}

\subsection{Proof of Proposition \ref{pro:Han}}

In this subsection we will finish the proof of Proposition \ref{pro:Han}. Using \eqref{eq:wvare<C}, \eqref{eq:wvare<R}, and Proposition \ref{pro:wvareLinftyBR<C}, we complete the proof of Proposition \ref{pro:Han} (i). Next, we prove  (ii) and (iii) of Proposition \ref{pro:Han}.

\begin{proof}[Proof of Proposition \ref{pro:Han} (ii).]
For convenience, we denote $\tilde{u}_{\varepsilon}(x):=\|u_{\varepsilon}\|_{L^{\infty}(\Omega)} u_{\varepsilon}(x)$. By the definition of $\mu_{\varepsilon}$ in \eqref{eq:defmuvare}, we have
\be\label{eq:eqtildeu}
\begin{cases}
(-\Delta)^{2} \tilde{u}_{\varepsilon} =\varepsilon\mu_{\varepsilon}^{-\frac{n-4}{2}} u_{\varepsilon}+\mu_{\varepsilon}^{-\frac{n-4}{2}} u_{\varepsilon}^{p}  & \text { in }\, \Omega, \\
\tilde{u}_{\varepsilon} =-\Delta \tilde{u}_{\varepsilon}=0  & \text { on }\, \partial \Omega.
\end{cases}
\ee
Note that
$$
\begin{aligned}
&\int_{\Omega} (\varepsilon\mu_{\varepsilon}^{-\frac{n-4}{2}} u_{\varepsilon}(x)+\mu_{\varepsilon}^{-\frac{n-4}{2}} u_{\varepsilon}^{p}(x)) \,\ud x\\
=&\varepsilon\mu_{\varepsilon}^{-\frac{n-4}{2}}\int_{\Omega} \mu_{\varepsilon}^{-\frac{n-4}{2}}v_{\varepsilon}(\mu_{\varepsilon}^{-1}x)\,\ud x+\mu_{\varepsilon}^{-\frac{n-4}{2}}\int_{\Omega}(\mu_{\varepsilon}^{-\frac{n-4}{2}}v_{\varepsilon}(\mu_{\varepsilon}^{-1}x))^p\,\ud x\\
=&\varepsilon\mu_{\varepsilon}^4\int_{\Omega_{\varepsilon}} v_{\varepsilon}(y)\,\ud y+\int_{\Omega_{\varepsilon}} v_{\varepsilon}^p(y)\,\ud y.
\end{aligned}
$$
By \eqref{eq:vvare<CV} and \eqref{eq:V(x)def}, we have
$$
\varepsilon\mu_{\varepsilon}^4\int_{\Omega_{\varepsilon}} v_{\varepsilon}(y)\,\ud y\leq C\varepsilon\mu_{\varepsilon}^4\int_{\Omega_{\varepsilon}} \Big(\frac{1}{1+c_0|y|^2}\Big)^{\frac{n-4}{2}}\,\ud y\leq C\varepsilon\mu_{\varepsilon}^4\int_{0}^{C\mu_{\varepsilon}^{-1}} r^3\,\ud r\leq C\varepsilon.
$$
On the other hand, given the uniform bound \eqref{eq:vvare<CV}, we use the Lebesgue dominated convergence theorem to obtain
\be\label{eq:Lebesgue-dominated}
\lim_{\varepsilon\to 0}\int_{\Omega_{\varepsilon}} v_{\varepsilon}^p(y)\,\ud y=\int_{\mathbb{R}^n} V^p(y)\,\ud y=c_0^{-\frac{n}{2}}\omega_n \int_{0}^{\infty} r^{n-1}(1+r^{2})^{-\frac{n+4}{2}} \,\ud r=\frac{2 c_0^{-{n}/{2}}\omega_n}{n(n+2)}.
\ee
Therefore,
$$
\lim_{\varepsilon\to 0}\int_{\Omega} (\varepsilon\mu_{\varepsilon}^{-\frac{n-4}{2}} u_{\varepsilon}(x)+\mu_{\varepsilon}^{-\frac{n-4}{2}} u_{\varepsilon}^{p}(x)) \,\ud x=\frac{2 c_0^{-{n}/{2}}\omega_n}{n(n+2)}.
$$
If $x \neq 0$, by \eqref{eq:Han-1} we have
$$
\begin{aligned}
&\varepsilon\mu_{\varepsilon}^{-\frac{n-4}{2}} u_{\varepsilon}(x) +\mu_{\varepsilon}^{-\frac{n-4}{2}} u_{\varepsilon}^{p}(x)\\
\leq& C\varepsilon\mu_{\varepsilon}^{-\frac{n-4}{2}}\Big(\frac{\mu_{\varepsilon}}{\mu_{\varepsilon}^{2}+|x|^{2}}\Big)
^{\frac{n-4}{2}}+C \mu_{\varepsilon}^{-\frac{n-4}{2}}\Big(\frac{\mu_{\varepsilon}}{\mu_{\varepsilon}^{2}+|x|^{2}}\Big)
^{\frac{n+4}{2}} \\
\leq& {C\varepsilon}{|x|^{4-n}}+{C \mu_{\varepsilon}^{4}}{|x|^{-n-4}} \to 0
\end{aligned}
$$
as $\varepsilon\to 0$.
Therefore we may conclude that
\be\label{eq:-Deltatildeuto}
(-\Delta)^{2} \tilde{u}_{\varepsilon}(x)=\varepsilon\mu_{\varepsilon}^{-\frac{n-4}{2}} u_{\varepsilon}(x) +\mu_{\varepsilon}^{-\frac{n-4}{2}} u_{\varepsilon}^{p}(x)\rightarrow \frac{2 c_0^{-{n}/{2}}\omega_n}{n(n+2)} \delta_{0}(x),
\ee
where $\delta_{0}$ is the Dirac function concentrating at $0$.

Let $\omega$ be a neighborhood of $\partial \Omega$ not containing $x=0$. It follows from \eqref{eq:eqtildeu}, Lemma \ref{lem:esti-unablau}, and the proof of \eqref{eq:-Deltatildeuto} that
$$
\|\Delta \tilde{u}_{\varepsilon}\|_{C^{1, \beta}(\omega)} \leq C(\|\varepsilon\mu_{\varepsilon}^{-\frac{n-4}{2}} u_{\varepsilon}+\mu_{\varepsilon}^{-\frac{n-4}{2}} u_{\varepsilon}^{p}\|_{L^{1}(\Omega)}+\|\varepsilon\mu_{\varepsilon}^{-\frac{n-4}{2}} u_{\varepsilon}+\mu_{\varepsilon}^{-\frac{n-4}{2}} u_{\varepsilon}^{p}\|_{L^{\infty}(\omega^{\prime})})\leq C,
$$
where $\beta \in(0,1)$, $\omega^{\prime}$ is a neighborhood of $\partial \Omega$ not containing $0$ such that $\omega \subset \subset \omega^{\prime}$ and $C>0$ independent of $\varepsilon$. Therefore, by Arzela-Ascoli's theorem, there exists a function $P(x)$ such that
$$
\Delta\tilde{u}_{\varepsilon}(x) \rightarrow P(x) \quad \text { in }\, C^{1, \alpha}(\omega),
$$
where $\alpha \in(0,1)$. Similar to the proof of \eqref{eq:uvareC1alpha}, we also have that there exists a function $Q(x)$ such that
$$
\tilde{u}_{\varepsilon}(x) \rightarrow Q(x) \quad \text { in }\, C^{1, \alpha}(\omega).
$$
By the Green representation formula, we have
$$
\tilde{u}_{\varepsilon}(x)=\int_{\Omega} G(x,y)(-\Delta)^2 \tilde{u}_{\varepsilon}(y)\,\ud y
$$
and
$$
\Delta\tilde{u}_{\varepsilon}(x)=\int_{\Omega} \Delta_xG(x,y)(-\Delta)^2 \tilde{u}_{\varepsilon}(y)\,\ud y.
$$
Using \eqref{eq:-Deltatildeuto}, we find
$$
Q(x)=\frac{2 c_0^{-n/2}\omega_{n}}{n(n+2)}G(x,0),\quad P(x)=\frac{2 c_0^{-n/2}\omega_{n}}{n(n+2)}\Delta_x G(x,0).
$$
This completes the proof of Proposition \ref{pro:Han} (ii).
\end{proof}

\begin{proof}[Proof of Proposition \ref{pro:Han} (iii).]
Multiplying the Pohozaev's identity \eqref{eq:Pohozaev} by $\|u_{\varepsilon}\|_{L^{\infty}(\Omega)}^{2}=\mu_{\varepsilon}^{4-n}$, we get
\begin{equation}\label{5.9}
2\varepsilon\mu_{\varepsilon}^{4-n}\int_{\Omega} u_{\varepsilon}^2 \,\ud x=\int_{\partial \Omega} \frac{\partial \tilde{u}_{\varepsilon}}{\partial \nu}\frac{\partial (-\Delta \tilde{u}_{\varepsilon})}{\partial \nu}(x \cdot \nu) \,\ud s.
\end{equation}
Similar to the proof of \eqref{eq:Lebesgue-dominated}, we have
$$
\begin{aligned}
2\varepsilon\mu_{\varepsilon}^{4-n}\int_{\Omega} u_{\varepsilon}^2(x) \,\ud x &=2\varepsilon\mu_{\varepsilon}^{8-n} \int_{\Omega_{\varepsilon}}v_{\varepsilon}^2(y)\,\ud y\\
&\to2\Big(\lim _{\varepsilon \rightarrow 0} \varepsilon\|u_{\varepsilon}\|_{L^{\infty}(\Omega)}^{\frac{2(n-8)}{n-4}}\Big)\int_{\mathbb{R}^{n}} V^{2}(y) \,\ud y\\
&=2c_0^{-n/2}\omega_n\Big(\lim _{\varepsilon \rightarrow 0} \varepsilon\|u_{\varepsilon}\|_{L^{\infty}(\Omega)}^{\frac{2(n-8)}{n-4}}\Big)
\int_{0}^{\infty}r^{n-1}(1+r^2)^{4-n}\,\ud r\\
&=c_1\Big(\lim _{\varepsilon \rightarrow 0} \varepsilon\|u_{\varepsilon}\|_{L^{\infty}(\Omega)}^{\frac{2(n-8)}{n-4}}\Big),
\end{aligned}
$$
where
$$
c_1=c_0^{-n/2}{\omega_n}\frac{\Gamma(\frac{n-8}{2}) \Gamma(\frac{n}{2})}{\Gamma(n-4)}.
$$
By \eqref{eq:Han-2} and \eqref{eq:Han-2-1}, the limit of the right hand side of \eqref{5.9} is equal to
$$
\frac{4 c_0^{-n}\omega_{n}^{2}}{n^{2}(n+2)^{2}} \int_{\partial \Omega} \frac{\partial G(x,0)}{\partial \nu}\frac{\partial (-\Delta_x G(x,0))}{\partial \nu}(x \cdot \nu) \,\ud s.
$$
As a result,
$$
\lim _{\varepsilon \rightarrow 0} \varepsilon\|u_{\varepsilon}\|_{L^{\infty}(\Omega)}^{\frac{2(n-8)}{n-4}}=\frac{4 c_0^{-n}\omega_{n}^{2}}{c_1n^{2}(n+2)^{2}} \int_{\partial \Omega} \frac{\partial G(x,0)}{\partial \nu}\frac{\partial (-\Delta_x G(x,0))}{\partial \nu}(x \cdot \nu) \,\ud s.
$$
Thanks to Lemma \ref{lem:PohozaevGreen}, we concludes the proof of Proposition \ref{pro:Han} (iii). We also know that the constant $C_n$ in the statement of the proposition is given by
\be\label{eq:constantCn}
C_n=\frac{2(n-4) c_0^{-n}\omega_{n}^{2}}{c_1n^{2}(n+2)^{2}}.
\ee
\end{proof}

\section{Proof of Theorem \ref{thm:2}}
This section is devoted to the proof of Theorem \ref{thm:2}. Here we follow the arguments in \cite{CA1999}, making the necessary modifications to handle the higher order case.

\begin{proof}[Proof of Theorem \ref{thm:2}]
Suppose, by contradiction, that there exist $u_{\varepsilon}$ and $v_{\varepsilon}$ satisfy \eqref{eq:Navier}, $u_{\varepsilon} \not \equiv v_{\varepsilon}$ for $\varepsilon \to 0$. Define
$$
\tilde{w}_{\varepsilon}(x):=u_{\varepsilon}(\mu_{\varepsilon}{x})-v_{\varepsilon}(\mu_{\varepsilon}{x}), \quad x \in \Omega_{\varepsilon},
$$
and
$$
w_{\varepsilon}(x):=\frac{\tilde{w}_{\varepsilon}(x)}{\|\tilde{w}_{\varepsilon}\|_{L^{\infty}(\Omega_{\varepsilon})}}
=\frac{\tilde{w}_{\varepsilon}(x)}{\|u_{\varepsilon}-v_{\varepsilon}\|_{L^{\infty}(\Omega)}}, \quad x \in \Omega_{\varepsilon},
$$
where $\mu_{\varepsilon}=\|u_{\varepsilon}\|_{L^{\infty}(\Omega)}^{-{2}/({n-4})}$ and $\Omega_{\varepsilon}=\Omega / \mu_{\varepsilon}$. Direct calculations show that $w_{\varepsilon}$ satisfies
$$
\begin{cases}
(-\Delta)^2 w_{\varepsilon}=(\varepsilon \mu_{\varepsilon}^4+c_{\varepsilon}(x)) w_{\varepsilon} & \text { in }\, \Omega_{\varepsilon}, \\
w_{\varepsilon}=-\Delta w_{\varepsilon}=0 & \text { on }\, \partial \Omega_{\varepsilon},
\end{cases}
$$
where
$$
c_{\varepsilon}(x)=p \int_{0}^{1}(t \mu_{\varepsilon}^{\frac{n-4}{2}}{u_{\varepsilon}}(\mu_{\varepsilon}{x})+(1-t) \mu_{\varepsilon}^{\frac{n-4}{2}}{v_{\varepsilon}}(\mu_{\varepsilon}{x}))^{p-1} \,\ud t.
$$
By \eqref{eq:vvaredef}-\eqref{eq:V(x)def}, we have
$$
\mu_{\varepsilon}^{\frac{n-4}{2}}{u_{\varepsilon}}(\mu_{\varepsilon}{x})\to V(x) \quad \text { uniformly on compact sets of }\, \mathbb{R}^{n}.
$$
On the other hand, by \eqref{eq:Han-3}, we have
$$
\lim _{\varepsilon \rightarrow 0}\|u_{\varepsilon}\|_{L^{\infty}(\Omega)}=\lim _{\varepsilon \rightarrow 0}\|v_{\varepsilon}\|_{L^{\infty}(\Omega)}=+\infty
$$
and
$$
\lim _{\varepsilon \rightarrow 0} \frac{\|u_{\varepsilon}\|_{L^{\infty}(\Omega)}}{\|v_{\varepsilon}\|_{L^{\infty}(\Omega)}}=1.
$$
Therefore,
$$
\|\mu_{\varepsilon}^{\frac{n-4}{2}}{v_{\varepsilon}}(\mu_{\varepsilon}\cdot)\|_{L^{\infty}(\Omega_{\varepsilon})}
=\frac{\|v_{\varepsilon}\|_{L^{\infty}(\Omega)}}{\|u_{\varepsilon}\|_{L^{\infty}(\Omega)}}\to 1\quad \text{ as }\, \varepsilon\to 0.
$$
So as in the proof of \eqref{eq:vvareeq}-\eqref{eq:v}, we also have
$$
\mu_{\varepsilon}^{\frac{n-4}{2}}{v_{\varepsilon}}(\mu_{\varepsilon}{x})\to V(x) \quad \text { uniformly on compact sets of }\, \mathbb{R}^{n}.
$$
Thus,
$$
c_{\varepsilon}(x) \rightarrow \frac{p}{(1+c_0|x|^{2})^{4}} \quad \text { uniformly on compact sets of }\, \mathbb{R}^{n}.
$$
By the assumptions of $\Omega$ and the method of moving planes, $w_{\varepsilon}$ is symmetric with respect to the hyperplanes $\{x_{i}=0\}$, $i=1, \cdots, n$. Since $\varepsilon \mu_{\varepsilon}^4 \rightarrow 0$ as $\varepsilon\to 0$ and $\|w_{\varepsilon}\|_{L^{\infty}(\Omega_{\varepsilon})}=1$, then it follows from standard elliptic theory that there exists a positive function $w$ such that (after passing to a subsequence) $w_{\varepsilon} \rightarrow w$ in $C_{loc}^4(\mathbb{R}^n)$, $w$ is symmetric and satisfies
\begin{equation}\label{9}
(-\Delta)^2 w=\frac{p}{(1+c_0|x|^{2})^{4}} w \quad \text { in }\, \mathbb{R}^{n}.
\end{equation}
Arguing as in \cite{CA1999}, it is easy to see that there exists a constant $C>0$ independent of $\varepsilon$, such that
\be\label{eq:DeltawvareL2<C}
\int_{\Omega_{\varepsilon}}|\Delta w_{\varepsilon}|^{2} \,\ud x \leq C,
\ee
and by Fatou's lemma, we get
$$
\int_{\mathbb{R}^n} |\Delta w|^{2} \,\ud x \leq C.
$$

We claim that there exist $C>0$ and $\delta>0$, such that
\begin{equation}\label{16}
|w_{\varepsilon}(x)| \leq C |y|^{4-n} \quad \text { in }\, \Omega_{\varepsilon} \backslash B_{\delta}.
\end{equation}
Indeed, define the Kelvin transform of $w_{\varepsilon}$:
$$
z_{\varepsilon}(x):=|x|^{4-n} w_{\varepsilon}\Big(\frac{x}{|x|^{2}}\Big), \quad x \in \Omega_{\varepsilon}^{*},
$$
where $\Omega_{\varepsilon}^{*}:=\{\frac{y}{|y|^{2}} \mid y \in \Omega_{\varepsilon}\}$. It will be enough to prove that $|z_{\varepsilon}(x)|$ is bounded in $\Omega_{\varepsilon}^{*}\cap B_R$ for some $R>0$. By calculations similar to \eqref{eq:wvare}, we obtain
$$
\begin{cases}
\displaystyle (-\Delta)^{2} z_{\varepsilon}=|x|^{-8}(\varepsilon\mu_{\varepsilon}^4+c_{\varepsilon}({x}/{|x|^{2}})) z_{\varepsilon} & \text { in }\, \Omega_{\varepsilon}^{*}, \\
\displaystyle z_{\varepsilon}=0,\, -\Delta z_{\varepsilon}=-4|x|^{-2}(x \cdot \nu) \frac{\partial z_{\varepsilon}}{\partial \nu} & \text { on }\, \partial \Omega_{\varepsilon}^{*}.
\end{cases}
$$
Then using \eqref{eq:Han-1} and Lemma \ref{lem:vare-muvare}, we can prove that $a(x):=|x|^{-8}(\varepsilon\mu_{\varepsilon}^4+c_{\varepsilon}({x}/{|x|^{2}})) \in L^{\alpha}(\Omega_{\varepsilon}^{*})$ for some $\alpha>n / 4$. On the other hand, using the Sobolev inequality \eqref{eq:Sobolevine} and \eqref{eq:DeltawvareL2<C}, we get
$$
\int_{\Omega_{\varepsilon}^{*} \cap B_{2 R}}|z_{\varepsilon}|^{p+1} \,\ud x  \leq \int_{\Omega_{\varepsilon}^{*}}|z_{\varepsilon}|^{p+1} \,\ud x=\int_{\Omega_{\varepsilon}}|w_{\varepsilon}|^{p+1} \,\ud y  \leq C\Big(\int_{\Omega_{\varepsilon}}|\Delta w_{\varepsilon}|^{2} \,\ud y\Big)^{(p+1) / 2} \leq C.
$$
Hence it follows from Lemma \ref{lem:localbound} that
$$
\sup _{\Omega_{\varepsilon}^{*} \cap B_R}|z_{\varepsilon}(x)| \leq C\Big(\int_{\Omega_{\varepsilon}^{*} \cap B_{2R}}|z_{\varepsilon}|^{p+1} \,\ud x\Big)^{1 /(p+1)} \leq C,
$$
and this proves the claim.

Now, we can back to \eqref{9}. It follows from \cite[Theorem 2.1]{BWWA2003} that there exist $a, b_i\in \mathbb{R}$, $i=1,\cdots,n$ such that
$$
w=a \frac{1-|x|^{2}}{(1+|x|^{2})^{(n-2) / 2}}+\sum_{i=1}^{n} b_{i} \frac{2 x_{i}}{(1+|x|^{2})^{(n-2) / 2}}.
$$
Since $w$ is a symmetric function with respect to the hyperplanes $\{x_{i}=0\}$, $i=1, \cdots, n$, each $b_{i}$ must be 0. On the other hand, using Lemma \ref{lem:Pohozaev} and Proposition \ref{pro:Han}, we can prove that $a=0$ by using a contradiction argument. We omit it since it is similarly to \cite[pages 108-111]{CA1999}. Therefore, we deduce that $w=0$.

Let $x_{\varepsilon}$ such that $\|w_{\varepsilon}\|_{L^{\infty}(\Omega_{\varepsilon})}=w_{\varepsilon}(x_{\varepsilon})=1$. Then we have $|x_{\varepsilon}| \rightarrow \infty$ because $w_{\varepsilon} \rightarrow w$ in $C_{loc}^4(\mathbb{R}^n)$ and $w=0$. This is impossible by \eqref{16} and Theorem \ref{thm:2} is proved.
\end{proof}

\end{document}